\documentclass[12pt,a4paper]{article}
\usepackage{hyperref}
\usepackage{tikz}
\usetikzlibrary{calc}
\usetikzlibrary{arrows}
\usepackage{amsmath}
\usepackage{amssymb}
\usepackage{amsthm}
\usepackage{framed}
\usepackage{multicol}
\usepackage[OT2,T1]{fontenc}
\usepackage{graphicx}
\usepackage[all]{xy}
\usepackage{color}
\usepackage{mathrsfs}
\usepackage{mathtools}

\numberwithin{equation}{section}

\newcommand{\sm}{\wedge}
\newcommand{\cof}{\mathrm{cof}}
\newcommand{\im}{\mathrm{im}}

\newcommand{\Mat}{\mathrm{Mat}}
\newcommand{\colim}{\mathrm{colim}}
\newcommand{\Hom}{\mathrm{Hom}}

\newcommand{\Sp}{\mathrm{Sp}}

\newcommand{\F}{\mathbb{F}}
\newcommand{\Z}{\mathbb{Z}}
\newcommand{\R}{\mathbb{R}}
\newcommand{\T}{\mathcal{T}}

\newcommand{\GL}{\mathrm{GL}}

\newcommand{\BP}{\mathrm{BP}}

\newcommand{\Aff}{\mathrm{Aff}}

\newcommand{\Po}{\mathcal{P}}
\newcommand{\Br}{\mathbf{B}}
\newcommand{\St}{\mathrm{St}}
\newcommand{\cd}{\mathrm{cd}}
\newcommand{\Pri}{\mathrm{Pr}}

\newcommand{\Top}{\mathrm{Top}}

\newcommand{\hocolim}{\mathrm{hocolim}}
\newcommand{\rreg}{\overline{\rho}}

\newcommand{\fun}{\mathrm{Fun}}
\newcommand{\cyc}{\mathrm{Cyc}}
\newcommand{\proj}{\mathrm{proj}}

\usepackage[utf8]{inputenc}

\title{Symmetric Powers and Eilenberg Maclane Spectra}
\author{Krishanu Sankar}
\date{}
\begin{document}
\maketitle
\begin{abstract}
We filter the equivariant Eilenberg Maclane spectrum $H\underline{\F}_p$ using the mod $p$ symmetric powers of the equivariant sphere spectrum, $\Sp_{\Z/p}^{\infty}(\Sigma^{\infty G}S^0)$. When $G$ is a $p$-group, we show that the layers in the filtration are the Steinberg summands of the equivariant classifying spaces of $(\Z/p)^n$ for $n=0, 1, 2, \ldots$. We show that the layers of the filtration split after smashing with $H\underline{\F}_p$. Along the way, we produced a general computation of the geometric fixed points of $H\underline{\Z}$ and $H\underline{\F}_p$ by using symmetric powers.
\end{abstract}
\newtheorem{theorem}{Theorem}
\newtheorem{lemma}[theorem]{Lemma}
\newtheorem{proposition}[theorem]{Proposition}
\newtheorem{corollary}[theorem]{Corollary}
\newtheorem{conjecture}[theorem]{Conjecture}

\theoremstyle{definition}
\newtheorem{definition}[theorem]{Definition}
\newtheorem{question}[theorem]{Question}
\newtheorem{comment}[theorem]{Comment}
\newtheorem{example}[theorem]{Example}

\theoremstyle{theorem}
\newtheorem{exe}{Proposition}

\section{Introduction}

For any abelian group $A$ and a pointed space $X$ let $A\otimes X$ be the space of finite $A$-linear combinations of points on $X$, with addition given by concatenation and the basepoint treated as 0. For any connected, pointed space $X$, the infinite symmetric power $\Sp^{\infty}(X)$ is weakly equivalent to $\Z\otimes X$, and so we write $\Sp_{\Z/p}^{\infty}(X):= \Z/p\otimes X$ for the mod $p$ symmetric powers.

The Eilenberg--Maclane space $K(A,\ell)$ is given by the space $A\otimes S^{\ell}$ (\cite{MC}). Stabilizing in $\ell$ gives the Eilenberg-Maclane spectrum $HA$ as the spectrum of $A$-linear combinations of points on the sphere spectrum $\Sigma^{\infty}S^0$. When the abelian group $A$ is a ring, we obtain the ring structure on $HA$ by linearly extending the product map $S^i\sm S^j \simeq S^{i+j}$.

Let $p$ be a prime, and let us specialize to the case $A=\Z/p=\F_p$, thought of as a ring. Henceforth, all spectra are understood to be $p$-localized. Then as above there is an equivalence of spectra $\Sp_{\Z/p}^{\infty}(\Sigma^{\infty}S^0)\simeq H\F_p$. The mod $p$ infinite symmetric power model has a filtration by finite symmetric powers. Let $\Sp_{\Z/p}^n(\Sigma^{\infty}S^0)\subset \Sp_{\Z/p}^{\infty}(\Sigma^{\infty}S^0)$ be the subspectrum
$$\Sp_{\Z/p}^n(\Sigma^{\infty}S^0)=\{a_1x_1 + \ldots+a_mx_m: a_i\in \{1, \ldots, p-1\}, x_i\in \Sigma^{\infty}S^0, \; \sum\limits_{i=1}^{m}a_i\le p^n\}.$$
The product restricts as $\Sp_{\Z/p}^{p^i}(\Sigma^{\infty}S^0)\sm \Sp_{\Z/p}^{p^j}(\Sigma^{\infty}S^0) \rightarrow \Sp_{\Z/p}^{p^{i+j}}(\Sigma^{\infty}S^0)$, and we consider the filtration of spectra
$$\Sigma^{\infty}S^0=\Sp_{\Z/p}^1(\Sigma^{\infty}S^0) \subset \Sp_{\Z/p}^p(\Sigma^{\infty}S^0)\subset \Sp_{\Z/p}^{p^2}(\Sigma^{\infty}S^0)\subset\cdots\subset \Sp_{\Z/p}^{\infty}(\Sigma^{\infty}S^0)\simeq H\F_p.$$
The above filtration was studied by Mitchell--Priddy, and (\cite{MP}, Theorem A) identifies the $n$-th layer of the filtration as the $n$-th suspension of the Steinberg summand of the classifying space $B(\Z/p)^n$. That is, there is an equivalence of spectra
$$\Sigma^n e_nB(\Z/p)^n_+ \simeq \Sp_{\Z/p}^{p^n}(\Sigma^{\infty}S^0)/\Sp_{\Z/p}^{p^{n-1}}(\Sigma^{\infty}S^0),$$
where $e_n$ is the topological idempotent corresponding to the idempotent element $e_n$ in the group ring $\Z_{(p)}[\GL_n(\F_p)]$. The aim of the present two-part paper is to prove an equivariant generalization of the theorem of Mitchell--Priddy.
\begin{theorem}\label{thm:maintheorem}
Let $G$ be a finite $p$-group. For each positive integer $n$, let $\Sp_{\Z/p}^n(\Sigma^{\infty G}S^0)$ be the $n$-th mod $p$ symmetric power on the genuine $G$-spectrum $\Sigma^{\infty G}S^0$. Then there is an equivalence of genuine $G$-spectra
$$\Sigma^n \mathbf{e}_nB_G(\Z/p)^n_+ \simeq \Sp_{\Z/p}^{p^n}(\Sigma^{\infty G}S^0)/\Sp_{\Z/p}^{p^{n-1}}(\Sigma^{\infty G}S^0).$$
\end{theorem}
The Steinberg summand construction $e_n$, and the $G$-equivariant classifying space construction $B_G\Z/p$ are both discussed in the companion paper, \cite{Sankar}. The $G$-equivariant classifying spaces $B_G\Z/p$ fit into a theory of equivariant principal bundles (\cite{GMM}), but in our situation they can be built explicitly as lens spaces (See \ref{example:equivariantlensspace}). The symmetric powers are topological in nature, while the Steinberg summand is algebraic, and Theorem \ref{thm:maintheorem} ties these two constructions together.

Note that by work of Lima-Filho and Dos Santos (\cite{LF}, \cite{DS}), there is an equivalence equivalence of $G$-spectra $\Sp_{\Z/p}^{\infty}(\Sigma^{\infty G}S^0) \simeq H\underline{\F}_p$. We also prove the following equivariant generalization of a well-known fact shown by Mitchell--Priddy.
\begin{theorem}\label{thm:eqsplitting}
Let $G$ be any finite $p$ group. The filtration $\{\Sp_{\Z/p}^{p^n}(\Sigma^{\infty G}S^0)\}_{n\ge 0}$ splits into its layers after smashing with $H\underline{\F}_p$. That is, there is an equivalence of $H\underline{\F}_p$-modules
$$H\underline{\F}_p\sm H\underline{\F}_p\simeq \bigvee\limits_{n\ge 0}H\underline{\F}_p\sm \Sigma^n \mathbf{e}_nB_G(\Z/p)^n_+.$$
\end{theorem}

\subsection{Intended application}
Our intended application is towards a computation of the mod $p$ equivariant dual Steenrod algebra $H\underline{\F}_p\sm H\underline{\F}_p$ when the group is $G=C_p$, the cyclic group of order $p$, and $p$ is an odd prime. Theorem \ref{thm:maintheorem} identifies the layers in the symmetric power filtration of the genuine Eilenberg-Maclane $G$-spectrum $H\underline{\F}_p$. Theorem \ref{thm:eqsplitting} tells us that after applying the functor $H\underline{\F}_p\sm (-)$, the filtration splits. One may then construct an equivariant cellular filtration of the $G$-space $B_G\Z/p$, and use this filtration to compute the summands $H\underline{\F}_p\sm \Sigma^n \mathbf{e}_nB_G(\Z/p)^n_+$ in an invariant-theoretic manner inspired by Mitchell--Priddy (\cite{MP}). This computation will appear in a joint paper with Dylan Wilson.

For the prime $p=2$, the $C_2$-equivariant dual Steenrod algebra has already been computed by Greenlees and later Hu--Kriz. Here, $\bigstar$ indicates a bigrading over the real representations of $C_2$, and $\sigma$ denotes the sign representation. Hu--Kriz produced generators in $\pi_{\bigstar}^{C_2}(H\underline{\F}_2\sm H\underline{\F}_2)$
\begin{itemize}
\item $\xi_1, \xi_2, \ldots$ of degree $|\xi_i|=(2^i-1)(1+\sigma)$, where $\xi_i$ comes from the $2^i(1+\sigma)$-cell of the $C_2$-space $\mathbf{CP}^{\infty}\simeq B_{C_2}S^{\sigma}$.
\item $\tau_0, \tau_1, \ldots$ of degree $|\tau_i|=2^i+(2^i-1)\sigma$, where $\tau_i$ comes from the $2^i(1+\sigma)$-cell of the $C_2$-space $B_{C_2}\Z/2$.
\end{itemize}
They then proved that
$$\pi_{\star}^{C_2}(H\underline{\F}_2\sm H\underline{\F}_2)\simeq (H\underline{\F}_2)_{\bigstar}[\xi_i, \tau_i]/\langle\tau_0a_{\sigma}=u_{\sigma}+\eta_R(u_{\sigma}), \tau_i^2=\tau_{i+1}a_{\sigma}+\xi_{i+1}\eta_R(u_{\sigma})\rangle$$
An exposition of this computation is given in \cite{LSWX}. It is curious that this result bears a similarity to the classical odd primary dual Steenrod algebra computed by Milnor. Hu--Kriz (\cite{HK}) used the above computations to analyze the Adams-Novikov spectral sequence for $\BP_{\mathbb{R}}$, which converges to the $2$-local stable homotopy groups of spheres.

\subsection{Historical significance}

The study of the mod $p$ cohomology of the symmetric power filtration traces back to Nakaoka (\cite{Na}), who studied the symmetric power filtration for $H\Z$. The $n$-th layer in this filtration is given a name,
$$\Sp^{p^n}(\Sigma^{\infty}S^0)/\Sp^{p^{n-1}}(\Sigma^{\infty}S^0)=:\Sigma^nL(n).$$
It is classically known that $L(n)\simeq e_n(B(\Z/p)^n)^{\rreg_n}$, i.e. that $L(n)$ is the Steinberg summand of the Thom spectrum on $B(\Z/p)^n$ corresponding to the reduced regular representation $\rreg_n:(\Z/p)^n \rightarrow O(p^n-1)$. See the paper of Arone-Dwyer-Lesh \cite{ADL} for a proof of this fact. The closely related spectra $M(n)\simeq e_nB(\Z/p)^n_+$ which (\cite{MP}) appear as the layers in the mod $p$ symmetric power were used by Mitchell provide a short proof of the Conner-Floyd conjecture (\cite{MitCF}). It is also a result of Welcher (\cite{Wel}) that the spectrum $M(n)$ has chromatic type $n$.

Let $H^*(-)$ denote mod $p$ cohomology. The mod $p$ symmetric power filtration of $H\F_p$ closely reflects the structure of the Steenrod algebra. Mitchell--Priddy showed in \cite{MP} that $H^*(\Sp_{\Z/p}^{p^n}(\Sigma^{\infty}S^0))$ has a basis given by the classes $\theta^I(u_n)$ where $u_n\in H^0(\Sp_{\Z/p}^{p^n}(\Sigma^{\infty}S^0))$ is a generator, and $I$ varies over admissible sequences of length at most $n$. Within this vector space, $H^*(\Sigma^n M(n))$ is the span of the classes $\theta^I(u_n)$ with length $\ell(I)=n$. A paper of Mitchell (\cite{Mit}) demonstrates the algebra and invariant theory involved.

The Whitehead conjecture, proven by Kuhn (\cite{KMP}, \cite{Ku}), states that the attaching maps yield a long exact sequence on homotopy groups
$$\xymatrix{\cdots\ar[r] & L(3)\ar[r] & L(2)\ar[r] & L(1)\ar[r] & S^0\ar[r] & H\Z\ar[r] & 0}$$
This is a resolution of the spectrum $H\Z$ by spacelike spectra, i.e. spectra which are stable summands of spaces (\cite{Ku3}). Kuhn's orignial and also his modern proof (\cite{Ku2}) utilizes the fact that there are almost no $\mathcal{A}$-module maps between the $H^*(L(n))$'s, which is a kind of rigidity result. It is an observation due to Mitchell--Priddy that $M(n) \simeq L(n)\vee L(n-1)$, and there is a mod $p$ version of the Whitehead conjecture. The main theorems of the present paper suggest that the above story may have an equivariant analogue.

\subsection{Outline of the paper}

Our proof of Theorem \ref{thm:maintheorem} is inspired by Mitchell--Priddy's proof of the nonequivariant analogue, which is by induction on $n$. The $n=1$ case, namely that $\Sigma e_1(B\Z/p)_+ \simeq \Sigma (B\Sigma_p)_+\simeq (H\F_p)_p/(H\F_p)_1$, is done by an explicit geometric argument. Now let $n\ge 2$. Both the Steinberg summands and the mod $p$ symmetric powers are equipped with product maps. But the product maps on Steinberg summands have \emph{retractions}. One has an intermediate inclusion of the Steinberg summand $e_nB(\Z/p)^n_+ \subset (e_1\boxtimes \cdots \boxtimes e_1)B(\Z/p)^n_+$. Now construct the commutative diagram
\begin{equation}\label{eqn:maindiagramintro}
\xymatrix{\Sigma^ne_nB(\Z/p)_+^n\ar[r]_{\subset}\ar@/^2ex/@{-->}[rr]^{\mathrm{id}} & \Sigma^n(e_1B\Z/p_+)^{\sm n}\ar[r]_{\mathrm{product}}\ar@{=}[d] & \Sigma^ne_nB(\Z/p)_+^n\\
& (\Sp_{\Z/p}^p(\Sigma^{\infty}S^0)/\Sp_{\Z/p}^1(\Sigma^{\infty}S^0))^{\sm n}\ar[r]_{\mathrm{product}}&\Sp_{\Z/p}^{p^n}(\Sigma^{\infty}S^0)/\Sp_{\Z/p}^{p^{n-1}}(\Sigma^{\infty}S^0)}.
\end{equation}

Mitchell--Priddy prove that the zigzag composition $\xymatrix{\Sigma^ne_nB(\Z/p)^n_+ \ar[r]& (H\F_p)_{p^n}/(H\F_p)_{p^{n-1}}}$ from the top left to the bottom right of this diagram is an isomorphism on mod $p$ cohomology, and therefore is a $p$-local equivalence. The proof is by induction on $n$ and requires a particular lemma about the interaction of the $e_n$ with the action of $\mathcal{A}$ in cohomology (\cite{MP}, Section 5).

Our proof of Theorem \ref{thm:maintheorem} follows a similar strategy. The $n=1$ case is proven in Section \ref{sec:firstcofiber}. For $n\ge 2$, we construct the analogous commutative diagram, and must prove that the zigzag composition $\Sigma^n \mathbf{e}_nB_G(\Z/p)_+^n \rightarrow \Sp_{\Z/p}^{p^n}(\Sigma^{\infty G}S^0)/\Sp_{\Z/p}^{p^{n-1}}(\Sigma^{\infty G}S^0)$ is a $p$-local equivalence of genuine $G$-spectra. By induction on the group $G$, it suffices to prove the zigzag composition is a mod $p$ homology isomorphism on geometric fixed point spectra. Most of the hard supporting work in the proof is in computing the geometric fixed point spectra $\Phi^G(\mathbf{e}_nB_G(\Z/p)_+^n)$ (done in \cite{Sankar}) and $\Phi^G(\Sp_{\Z/p}^{p^n}(\Sigma^{\infty G}S^0)/\Sp_{\Z/p}^{p^{n-1}}(\Sigma^{\infty G}S^0)$, and computing the effect of the maps in the above diagram.

Here is an outline of this paper. In Section \ref{sec:symmetricpowers} we define the notion of the \emph{primitives of a $G$-space $X$}, and prove a wedge sum decomposition of the geometric fixed points of the mod $p$ symmetric powers of a suspension $G$-spectrum $\Sigma^{\infty G}X$ whose summands depend on the fixed point space $X^G$ and the primitives of the $G$-space $S^{\infty\rreg_G}$. In Section \ref{sec:stableprimitives}, we prove that the primitives of $S^{\infty\rreg_G}$ are the homotopy orbit space of a certain subgroup complex. In Section \ref{sec:productsonprimitives}, we compute the effect of the product maps relating the geometric fixed point spectra $\{\Phi^G\Sp_{\Z/p}^n(\Sigma^{\infty G}S^0)\}_{n\ge 0}$ as they pertain to our decomposition of the geometric fixed points.

In Section \ref{sec:firstcofiber}, we prove that there is an equivalence of $G$-spectra
$$\Sp_{\Z/p}^p(\Sigma^{\infty G}S^0)/\Sp_{\Z/p}^1(\Sigma^{\infty G}S^0) \simeq \Sigma^{\infty G}(B_G\Aff_1)_+$$
where $\Aff_1$ is the group of $p(p-1)$ affine permutations of the one-dimensional $\F_p$-line. This is the $n=1$ case of Theorem \ref{thm:maintheorem}. In Section \ref{sec:modpsymmetricpowersandSteinbergsummands}, we combine the cumulative results of \cite{Sankar} and Sections \ref{sec:symmetricpowers}, \ref{sec:stableprimitives}, \ref{sec:productsonprimitives}, to describe the zigzag composition of Diagram \ref{eqn:maindiagramintro} on the homology of the geometric fixed points
$$f:H_*(\Phi^G(\Sigma^n \mathbf{e}_nB_G(\Z/p)_+^n);\F_p) \rightarrow H_*(\Phi^G(\Sp_{\Z/p}^{p^n}(\Sigma^{\infty G}S^0)/\Sp_{\Z/p}^{p^{n-1}}(\Sigma^{\infty G}S^0));\F_p).$$
We then use matrix algebra to prove that this map is an isomorphism, completing the proof of Theorem \ref{thm:maintheorem}. Finally in Section \ref{sec:splittingofthefiltration} we prove Theorem \ref{thm:eqsplitting} using a short and straightforward argument.

\section{Fixed Points in Symmetric Powers}
\label{sec:symmetricpowers}
Let $G$ be a finite group and let $X$ be a pointed $G$-space. We decompose the $G$-fixed points of the infinite symmetric power $\Sp^{\infty}(X)$ by using a tower of fibrations. Specifically, we produce in Proposition \ref{prop:symmetricpowersfixedpoints} a finite sequence of topological abelian monoids
$$\xymatrix{\Sp^{\infty}(X^G)=F_0\ar[r] & F_1\ar[r] & F_2\ar[r] & \cdots\ar[r] & F_m=\Sp^{\infty}(X)^G}$$
and identify $F_{i-1} \rightarrow F_i$ as the homotopy fiber of a map from $F_i$ to the infinite symmetric power of what we call a \emph{primitive} of the $G$ action on $X$, or $\Pri^{G/H}(X)$. There is a single primitive functor, and hence a single layer, for each conjugacy class of subgroups of $G$.

In Proposition \ref{prop:modpsymmetricpowersfixedpoints}, we generalize this decomposition to the free $\Z/p$--module generated by $X$, i.e. the mod--$p$ symmetric powers, the case of Proposition \ref{prop:symmetricpowersfixedpoints} being $p=0$. The main content of Proposition \ref{prop:modpsymmetricpowersfixedpoints} is that this decomposition splits if $G$ is a $p$-group.

In Proposition \ref{prop:stablefinitesymmetricpowers}, we extend this decomposition to the case of the geometric fixed points $\Phi^G$ of the infinite symmetric power of a suspension $G$-spectrum $\Sigma^{\infty G}X$. In the context of spectra, we may exploit two further phenomena: fiber and cofiber sequences coincide, and the natural (nonequivariant) map
$$X\sm \Sp^n(\Sigma^{\infty}S^0) \rightarrow \Sp^n(\Sigma^{\infty}X)$$
is an equivalence. This allows us to decompose the geometric fixed points of the finite symmetric power
$$\Phi^G\Sp^n(\Sigma^{\infty G}X)$$
in a similar manner to our decomposition for $\Sp^{\infty}$. This is Proposition \ref{prop:stablefinitesymmetricpowers}. Combining Proposition \ref{prop:stablefinitesymmetricpowers} with Proposition \ref{prop:modpsymmetricpowersfixedpoints}, we obtain a decomposition of
$$\Phi^G\Sp_{\Z/p}^n(\Sigma^{\infty G}X)$$
when $G$ is a $p$-group. This is Proposition \ref{prop:modpstablefinitesymmetricpowers}. The infinite symmetric power of the genuine $G$-spectrum $\Sigma^{\infty G}S^0$ is a model for the equivariant Eilenberg-Maclane spectrum $H\underline{\Z}$, and therefore the finite symmetric powers are a filtration for this genuine $G$ spectrum (Corollary \ref{cor:equivariantEilenbergMacLane}). We thus deduce decompositions of the geometric fixed point spectra $\Phi^G(H\underline{\Z})$ and $\Phi^G(H\underline{\F}_p)$ in Propositions \ref{prop:geometricfixedpointsofHZ} and \ref{prop:geometricfixedpointsofHF_p}.

\subsection{Symmetric Powers with Coefficients}
\label{subsec:symmetricpowerswithcoefficients}
\begin{definition}\label{definition:symmetricpowers}
Let $X$ be a pointed $G$-space, and let $n\ge 1$ be a positive integer. The \emph{$n$-th symmetric power} of $X$ is the pointed $G$-space
$$\Sp^n(X) = X^{\times n}/\Sigma_n.$$
There are inclusions $\Sp^{n-1}(X) \hookrightarrow \Sp^n(X)$ given by adding a copy of the basepoint of $X$. The quotient of this inclusion is given by
$$\Sp^{n-1}(X) \rightarrow \Sp^n(X) \rightarrow X^{\sm n}/\Sigma_n.$$
The \emph{infinite symmetric power} of $X$ is the colimit
$$\Sp^{\infty}X:=\colim_{n\to\infty}\Sp^n(X)=\{(x_1+x_2+\ldots+x_n): x_1, \ldots, x_n\in X\}.$$
The space $\Sp^{\infty}(X)$ is the free topological abelian monoid on $X$, and we write its elements as formal sums as above.
\end{definition}
It is easily seen that $\Sp^n(-)$ is a functor from pointed $G$-spaces to pointed $G$-spaces enjoying the following properties.
\begin{enumerate}
\item There is an \emph{addition} map given by formally summing,
$$\Sp^m(X) \times \Sp^n(X) \rightarrow \Sp^{m+n}(X).$$
\item Because the addition map is commutative, there is an \emph{amalgamation} map,
$$\Sp^m(\Sp^n(X)) \rightarrow \Sp^{mn}(X).$$
\item There is a \emph{multiplication} map given by formally multiplying,
$$\Sp^m(X) \sm \Sp^n(Y) \rightarrow \Sp^{mn}(X\sm Y).$$
\end{enumerate}

More generally if $M$ is any $G$-module, one may define the free topological $M$-module over $X$, denoted by $M\otimes X$ (see definition 2.1 in \cite{DS} for the precise definition).
$$M\otimes X=\{m_1x_1+\ldots+m_nx_n:m_i\in M, x_i\in X\}$$

The group $G$ acts on the space $M\otimes X$ by acting on both coordinates simultaneously. When $M$ is the group $\Z$ with trivial $G$-action, then the inclusion $\Sp^{\infty}(X) \hookrightarrow \Z\otimes X$ is a weak equivalence as long as $X^H$ is connected for every subgroup $H$.

When $M=\Z/p$ with trivial action, we given a special name to this functor.
\begin{definition}\label{definition:modpsymmetricpowers}
For any pointed $G$-space $X$, define the \emph{infinite mod $p$ symmetric power} of $X$ as the topological $\Z/p$-module,
$$\Sp^{\infty}_{\Z/p}:=\Z/p\otimes X=\{(a_1x_1+\ldots+a_nx_n:a_i\in \Z/p, x_i\in X\}.$$
There is an obvious surjection $\Sp^{\infty}(X) \rightarrow \Sp_{\Z/p}^{\infty}(X)$. We define the \emph{$n$-th mod $p$ symmetric power}, denoted $\Sp_{\Z/p}^n(X)$, as the image of $\Sp^n(X)$ under this surjection. Note that $\Sp^n(X)=\Sp_{\Z/p}^n(X)$ for $1\le n\le p-1$.
\end{definition}

It is easily seen that for any $G$-module $A$, the functor $A\otimes (-)$ takes cofiber sequences of pointed $G$-spaces to fiber sequences of unpointed $G$-spaces.

\subsection{The Primitives of a $G$-space}
\label{subsec:theprimitivesofaGspace}
Let $G$ be a finite group, and let $X$ be a $G$-space. For any two subgroups $H\subseteq K \subseteq G$, there is an inclusion of fixed point spaces $X^K \subseteq X^H$. Thus,
\begin{definition}
The fixed point spaces $\{X^H\}_{H\subseteq G}$ define a filtration of $X$ indexed over the (opposite) poset of subgroups of $G$. We call this filtration the \emph{fixed point filtration} of $X$.
\end{definition}

The space $X^H$ carries an action of the normalizer subgroup $N_G(H)=\{g\in G: gHg^{-1}=H\}$. Since $H$ acts trivially on $X^H$, we obtain a residual action of the \emph{Weyl group}, $W_G(H)=N_G(H)/H$ upon $X^H$.
\begin{definition}\label{definition:primitives}
We let $X(H)$ denote the $H$-th layer in the fixed point filtration, i.e.
$$X(H):=X^H/(\bigcup\limits_{K\supsetneq H}X^K).$$
By definition, each point $x\in X(H)$ has isotropy group $\{g\in G:gx=x\}=H$. Therefore, the Weyl group $W_G(H)$ acts freely on the pointed space $X(H)$, and we define the quotient of this action as the \emph{$G/H$-primitives of $X$}
$$\Pri^{G/H}(X):= X(H)/W_G(H).$$
\end{definition}
\begin{example}
The first stage of the fixed point filtration is $\Pri^{G/G}(X)=X^G$.
\end{example}
For each subgroup $H$, $\Pri^{G/H}(-)$ is a functor from pointed $G$-spaces to pointed spaces enjoying the following properties.
\begin{enumerate}
\item Suppose that $H, H'\subseteq G$ are conjugate, that is there is some $g\in G$ such that $H'=gHg^{-1}$. Left action by $g$ induces a homeomorphism
$$g:X(H) \rightarrow X(H').$$
and therefore, a homeomorphism $\xymatrix{\Pri^{G/H}(X) \ar[r]^{\cong}& \Pri^{G/H'}(X)}$.
\item Let $[H]$ denote the set of subgroups of $G$ which are conjugate to $H$. Then $G$ acts on the wedge sum $\bigvee\limits_{H'\in [H]}X(H')$, and the isotropy group of $x\in X(H')$ is $H'$. One may just as well define $\Pri^{G/H}(X)$ as the strict quotient $\Pri^{G/H}(X)= (\bigvee\limits_{H'\in [H]}X(H'))/G$.

\item (Product) For $i=1, 2$, let $G_i$ be a finite group, $H_i\subseteq G_i$ a subgroup, and $X_i$ a pointed $G_i$-space. Then there is a \emph{product} map, which is a homeomorphism
$$\Pri^{G_1/H_1}(X_1)\sm \Pri^{G_2/H_2}(X) \cong \Pri^{\frac{G_1\times G_2}{H_1\times H_2}}(X_1\sm X_2).$$
\end{enumerate}

\subsection{Subgroup Orderings}
\label{subsec:subgrouporderings}
We would like to convert the fixed point filtration of $X$ into a true filtration, and this motivates the following definition.

\begin{definition}\label{definition:subgroupordering}
A \emph{subgroup ordering for $G$} is a sequence of subgroups $H_0, H_1, \ldots, H_m\subseteq G$ satisfying two properties.
\begin{enumerate}
\item Every subgroup of $G$ is conjugate to exactly one of the $H_i$.
\item If $i<j$, then no conjugate of $H_i$ is a subgroup of $H_j$.
\end{enumerate}
These two properties imply that $H_0=G$ and $H_m=\{1\}$.
\end{definition}
Given a subgroup ordering $H_0, \ldots, H_m$, one may construct an honest filtration of $X$,
\begin{equation}\label{equation:isotropyfiltration}
X^G=X^{H_0}\subseteq \bigcup\limits_{H'\in [H_1]}X^{H'}\subseteq \bigcup\limits_{H'\in [H_2]}X^{H'}\subseteq \cdots \subseteq \bigcup\limits_{H'\in [H_{m-1}]}X^{H'} \subseteq X^{H_m}=X.
\end{equation}
It is clear that when $G$ is a finite group, a subgroup ordering exists, and henceforth we will simply pick one. For our purposes, it does not matter which one.

\subsection{Fixed Points and Primitives}
\label{subsec:fixedpointsandprimitives}
Recall that $\Sp^{\infty}(X)$ denotes the free abelian monoid on the pointed $G$-space $X$. The $G$-fixed point space $\Sp^{\infty}(X)^G \subseteq \Sp^{\infty}(X)$ is the space of sums $(x_1+\ldots+x_n)$ such that $G$ permutes the points $x_1, \ldots, x_n$ in some fashion. In general, the topological abelian monoid $\Sp^{\infty}(X)^G$ is not a free one, i.e. there is no pointed space $Y$ such that $\Sp^{\infty}(X)^G\cong \Sp^{\infty}(Y)$.

We will use the fixed point filtration of $X$ to product a resolution of $\Sp^{\infty}(X)^G$ by the free abelian monoids $\Sp^{\infty}(\Pri^{G/H_i}(X))$ (Proposition \ref{prop:symmetricpowersfixedpoints}). We will also show that the analogous resolution of $\Sp_{\Z/p}^{\infty}(X)^G$ splits as a Cartesian product when $G$ is a $p$-group (Proposition \ref{prop:modpsymmetricpowersfixedpoints}). Our proofs will rely on the following lemma.

\begin{lemma}\label{lemma:oneisotropygroup}
Let $H\subseteq G$ be a subgroup, and suppose that every point $x\in X$ has isotropy group conjugate to $H$. Let $[G:H]=|G|/|H|$ denote the index of the subgroup $H$. Then
\begin{enumerate}
\item The fixed point space $\Sp^{\infty}(X)^G$ is the free abelian monoid $\Sp^{\infty}(X)^G \cong \Sp^{\infty}(\Sp^{[G:H]}(X)^G)$, with the homeomorphism given by amalgamation.
\item There is a homeomorphism $\Sp^{[G:H]}(X)^G \cong \Pri^{G/H}(X)$.
\end{enumerate}

\end{lemma}
\begin{proof}
(1) is immediate from the observation that any sum $(x_1+x_2+\ldots+x_n)\in \Sp^{\infty}(X)^G$ decomposes as a sum of $G$-orbits isomorphic to $G/H$.

Now we prove (2). Let $a=[N_G(H):H]$, and let $b=[G:N_G(H)]$, so that $[G:H]=ab$. By definition, $\Pri^{G/H}(X)= X^H/W_G(H)$. The transfer map
$$X^H/W_G(H) \rightarrow \Sp^{[G:H]}(X)^G$$
$$[x] \mapsto \sum\limits_{g \in G/H}gx$$
has an inverse given as follows. If $(x_1+\ldots+x_{[G:H]})$ is a $G$-fixed point of $\Sp^{[G:H]}(X)$, then there are $a$ points $x_j$ with isotropy group $H$ --- pick one. The image of $x_j$ in the quotient space $X^H/W_G(H)$ does not depend on which point we pick, and
$$\sum\limits_{g \in G/H}gx_j=(x_1+\ldots+x_{[G:H]}).$$
\end{proof}

\begin{proposition}\label{prop:symmetricpowersfixedpoints}
Let $H_0, H_1, \ldots, H_m$ be a subgroup ordering (Definition \ref{definition:subgroupordering}) for $G$.  Then the functor $(\Sp^{\infty}(-))^G$ has a filtration by functors valued in topological abelian monoids
$$\xymatrix{\Sp^{\infty}((-)^G)=F_0\ar[r] & F_1\ar[r] & F_2\ar[r] & \cdots\ar[r] & F_{m-1}\ar[r] & F_m= (\Sp^{\infty}(-))^G}$$
where for each $1\le i\le m$, the map $F_{i-1} \rightarrow F_i$ is the homotopy fiber of a fibration $F_i \rightarrow \Sp^{\infty}(\Pri^{G/H_i}(-))$.

\end{proposition}
\begin{proof}
Fix a pointed $G$-space $X$. For $0\le i\le n$, let $C_i$ denote the subspace $C_i=\bigcup\limits_{H'\in [H_i]}X^{H'}$. Then as in Equation \ref{equation:isotropyfiltration}, one has the following filtration of $X$:
$$X^G=C_0 \subseteq C_1 \subseteq \cdots\subseteq C_{m-1}\subseteq C_m=X.$$
For each $0\le i\le n$, let $F_i:=\Sp^{\infty}(C_i)^G$. The functor $\Sp^{\infty}(-)$ (Definition \ref{definition:symmetricpowers}) takes cofiber sequences to fiber sequences, and the functor $(-)^G$ preserves fiber sequences. Thus, for each $1\le i \le n$, one has a fiber sequence
$$F_{i-1} \rightarrow F_i \rightarrow \Sp^{\infty}(C_i/C_{i-1})^G.$$
Every isotropy group of the cofiber $C_i/C_{i-1}$ is conjugate to $H_i$. Thus, by Lemma \ref{lemma:oneisotropygroup}, there is a homeomorphism
$$\Sp^{\infty}(C_i/C_{i-1})^G \cong \Sp^{\infty}(\Pri^{G/H_i}(C_i/C_{i-1})).$$
It is easily seen that $\Pri^{G/H_i}(C_i/C_{i-1})=\Pri^{G/H_i}(X)$, by the definition of the primitive functor (\ref{definition:primitives}). Thus, the proof is complete.
\end{proof}

\begin{center}
\begin{tikzpicture}

\draw (0,0) ellipse (6cm and 3cm);
\node [above] at (0,3) {\Large $X$};

\node[draw,circle] (1) at (-4,1.5) {$x_1$};
\node[draw,circle] (2) at (-5,0.5) {$x_2$};
\node[draw,circle] (3) at (-3.75,-1) {$x_3$};
\node[draw,circle] (4) at (-2.75,0) {$x_4$};

\node[draw,circle] (5) at (-1,2) {$x_5$};
\node[draw,circle] (6) at (1,2) {$x_6$};

\node[draw,circle] (7) at (-0,-1.5) {$x_7$};
\node[draw,circle] (8) at (1.5,-2) {$x_8$};
\node[draw,circle] (9) at (3,-1) {$x_9$};

\draw [->] (1) to [out=180,in=90] (2);
\draw [->] (2) to [out=270,in=180] (3);
\draw [->] (3) to [out=0,in=270] (4);
\draw [->] (4) to [out=90,in=0] (1);

\draw [->] (5) to [out=45,in=135] (6);
\draw [->] (6) to [out=225,in=315] (5);

\draw [->] (7) to [out=315,in=180] (8);
\draw [->] (8) to [out=0,in=240] (9);
\draw [->] (9) to [out=150,in=60] (7);

\end{tikzpicture}
\vskip 0.1in
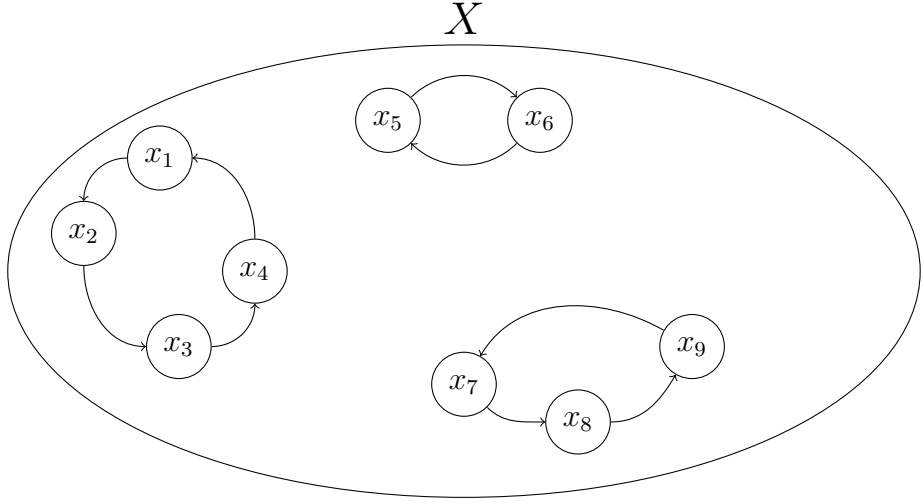
Pictured: An element of $(\Sp^9(X))^G$ with three orbits.
\end{center}

\begin{proposition}\label{prop:modpsymmetricpowersfixedpoints}
Let $G$ be a $p$-group. Let $H_0, H_1, \ldots, H_m$ be a subgroup ordering (\ref{definition:subgroupordering}) for $G$.  Then the functor $(\Sp_{\Z/p}^{\infty}(-))^G$ has a filtration by functors valued in topological abelian monoids
$$\xymatrix{\Sp_{\Z/p}^{\infty}((-)^G)=F_0\ar[r] & F_1\ar[r] & F_2\ar[r] & \cdots\ar[r] & F_{m-1}\ar[r] & F_m= (\Sp_{\Z/p}^{\infty}(-))^G}$$
where for each $1\le i\le m$, the map $F_{i-1} \rightarrow F_i$ is the homotopy fiber of a split fibration
$$\xymatrix{F_{i-1}\ar[r] & F_i\ar[r]& \Sp_{\Z/p}^{\infty}(\Pri^{G/H_i}(-))\ar@/_2ex/@{-->}[l]}.$$
Thus, the functor $(\Sp_{\Z/p}^{\infty}(-))^G$ splits as a Cartesian product
$$(\Sp_{\Z/p}^{\infty}(-))^G\cong \prod\limits_{i=0}^{m}\Sp_{\Z/p}^{\infty}(\Pri^{G/H_i}(-)).$$
\end{proposition}
\begin{proof}
Fix a pointed $G$-space $X$. For $0\le i\le n$, let $C_i$ denote the subspace $C_i=\bigcup\limits_{H'\in [H_i]}X^{H'}$ as in Proposition \ref{prop:symmetricpowersfixedpoints}. Let $F_i=(\Sp_{\Z/p}^{\infty}(C_i))^G$. Consider the continuous map of pointed spaces
$$f:X^{H_i} \rightarrow F_i$$
defined by $f(x)=\sum\limits_{g\in G/H_i}gx$. Observe the following two properties of $f$.
\begin{itemize}
\item If the isotropy group of $x$ is some $K\supsetneq H_i$, then the index $[K:H_i]$ is a power of $p$ (because $G$ is a $p$-group), and it follows that $f(x)=0$. Therefore, the map $f$ factors through the quotient $X^{H_i}/\bigcup\limits_{K\supsetneq H_i}X^K$.
\item It is easily checked that for any $w$ in the Weyl group $W_G(H_i)$, we have $f(wx)=f(x)$.
\end{itemize}
So $f$ descends to a map
$$\tilde{f}:\Pri^{G/H_i}(X)=(X^{H_i}/\bigcup\limits_{K\supsetneq H_i}X^K)/W \rightarrow (\Z/p\otimes \bigcup\limits_{H\sim H_i}X^H)^G$$
Extend this map $\Z/p$-linearly to obtain the desired section map $(\Sp_{\Z/p}^{\infty}(\Pri^{G/H_i}(X))) \rightarrow F_i$. It is easily seen that this map is the inverse to the fibration $F_i \rightarrow (\Sp_{\Z/p}^{\infty}( \Pri^{G/H_i}(X)))$.
\end{proof}

As a corollary, we deduce the existence of another structure map for the primitives functor when $G$ is a finite $p$-group. Let $G'\subseteq G$ be any subgroup, and let $H\unlhd G$ be a normal subgroup. Let $H'=G'\cap H$ denote the intersection. Then the $G$-set $G/H$ decomposes as a $G'$-set into the disjoint union of $p^m=\frac{|G||H'|}{|G'||H|}$ copies of $G'/H'$, i.e.
$$\mathrm{res}_{G'}^G(G/H) = (G'/H')\sqcup (G'/H')\sqcup\cdots \sqcup (G'/H').$$
The reason for assuming that $H\unlhd G$ is to avoid subtlety about how $gHg^{-1}\cap G'$ may vary in size for different conjugate subgroups $gHg^{-1}$.

For any $G$-space $X$, consider the inclusion of fixed points
$$\iota: \Sp_{\Z/p}^{\infty}(X)^G \hookrightarrow \Sp_{\Z/p}^{\infty}(X)^{G'}.$$
Using the Cartesian product decomposition of Proposition \ref{prop:modpsymmetricpowersfixedpoints}, we obtain a $\Z/p$-module map $\Sp_{\Z/p}^{\infty}(\Pri^{G/H}(X)) \rightarrow \Sp_{\Z/p}^{\infty}(\Pri^{G'/H'}(X))$ which is the free $\Z/p$-module on a map
$$\Pri^{G/H}(X) \rightarrow \Sp_{\Z/p}^{p^m}(\Pri^{G'/H'}(X)).$$
\begin{definition}\label{definition:primitivesrestriction}
The map $\Pri^{G/H}(X) \rightarrow \Sp_{\Z/p}^{p^m}(\Pri^{G'/H'}(X))$ described above is called the \emph{restriction map} for primitives.
\end{definition}

\subsection{Passing to Spectra}
\label{subsec:passingtospectra}
For our purposes, a \emph{spectrum} $\mathbf{X}$ is a sequence of pointed spaces $X_0, X_1, X_2, \ldots$ with specified structure maps $\Sigma X_n \rightarrow X_{n+1}$. A map between spectra $f:\mathbf{X} \rightarrow \mathbf{Y}$ is a collection of maps $f_n:X_n \rightarrow Y_n$ compatible with the structure maps. Such a map $f$ is an \emph{equivalence} if $f_n$ is an equivalence on homotopy groups up to dimension $n+\ell(n)$ for every $n$, where $\ell:\Z_{\ge 0} \rightarrow \Z_{\ge 0}$ is some function such that $\lim\limits_{n\to \infty}\ell(n)=\infty$. Intuitively, this means that as $n$ grows, the connectivity of $f_n$ grows at a rate faster than $n$.

When $G$ is a finite or compact Lie group, there are two types of $G$-spectra to consider. See \cite{HHR} for a detailed working guide to equivariant spectra.
\begin{definition}
A \emph{na{\"i}ve $G$-spectrum} $\mathbf{X}$ is a sequence of pointed $G$-spaces $X_0, X_1, X_2, \ldots$ with specified $G$-equivariant maps $\Sigma X_n \rightarrow X_{n+1}$. For example, if $X$ is a pointed $G$-space, then its na{\"i}ve suspension spectrum is the na{\"i}ve $G$-spectrum
$$\Sigma^{\infty}X=\{S^n\sm X\}_{n\ge 0}.$$
For every subgroup $H\subseteq G$, one has the \emph{fixed point spectrum} $\mathbf{X}^H=\{X_n^H\}_{n\ge 0}$. A map $f:\mathbf{X} \rightarrow \mathbf{Y}$ of na{\"i}ve $G$-spectra is said to be an \emph{equivalence} if $f^H:\mathbf{X}^H \rightarrow \mathbf{Y}^H$ is an equivalence of spectra for every subgroup $H\subseteq G$.
\end{definition}

\begin{definition}
Let $RO(G)$ denote the orthogonal representation ring of $G$, and let $\rho_G$ (resp. $\rreg_G$) denote the regular (resp. reduced regular) real representation of $G$. A \emph{genuine $G$-spectrum} $\mathbf{X}$ is a collection of pointed $G$-spaces $\{X_V\}$ for every orthogonal $G$-representation $V$, with specified $G$-equivariant maps
$$S^{W-V}\sm X_V \rightarrow X_W$$
whenever $V$ is a subrepresentation of $W$. For example, if $X$ is a pointed $G$-space, then its \emph{genuine suspension spectrum} is the genuine $G$-spectrum
$$\Sigma^{\infty G}X=\{S^V\sm X\}_{V\in RO(G)}.$$
The \emph{geometric fixed point spectrum}, $\Phi^G\mathbf{X}$, is the spectrum whose $n$-th space is defined by
$$(\Phi^G\mathbf{X})_n=X_{n\rho_G}^G.$$
For each subgroup $H\subseteq G$, let $i_H^*\mathbf{X}$ denote the genuine $H$-spectrum given by restricting the action, and define $\Phi^H\mathbf{X}:= \Phi^Gi_H^*\mathbf{X}$. A map $f:\mathbf{X} \rightarrow \mathbf{Y}$ of genuine $G$-spectra is said to be an \emph{equivalence} if $\Phi^H i_H^* f:\Phi^H i_H^*\mathbf{X} \rightarrow \Phi^H i_H^*\mathbf{Y}$ is an equivalence of spectra for every subgroup $H\subseteq G$. Note that of $X$ is a pointed $G$-space, then $\Phi^G(\Sigma^{\infty G}X)=\Sigma^{\infty}(X^G)$.
\end{definition}
\begin{definition}\label{definition:promotingGspectra}
Any na{\"i}ve $G$-spectrum $\mathbf{X}=\{X_n\}_{n\ge 0}$ may be promoted to a genuine $G$-spectrum $\{X_V\}_{V\in RO(G)}$ as follows. For any orthogonal representation $V$, let $V^G$ denote the fixed points and $\overline{V}\subset V$ denote an orthogonal complement to $V^G$. Note that $\overline{V}^G=0$. Then we define pointed $G$-spaces $X_V$ indexed over the orthogonal representations $V$ by
$$X_V:= S^{\overline{V}}\sm X_{V^G}.$$
This construction defines a functor $i_*$ from na{\"i}ve $G$-spectra to genuine $G$-spectra, and this functor can be viewed as inverting the representation sphere $S^{\rreg_G}$. It is easily checked that
$$\Phi^G(i_*\mathbf{X})\simeq \mathbf{X}^G.$$

\end{definition}

Let $n$ be a positive integer. The $n$-th symmetric power functor $\Sp^n(-)$ may be applied to a na{\"i}ve or genuine $G$-spectrum by application to each component space. We decompose the geometric fixed points of the $n$-th symmetric power of a genuine suspension spectrum of a pointed $G$-space $X$ via a filtration of spectra
$$\xymatrix{\Sp^n(\Sigma^{\infty}X^G)=F_0\ar[r] & F_1\ar[r] & \cdots\ar[r] & F_m=\Phi^G(\Sp^n(\Sigma^{\infty G}X))}$$
where the cofiber of the map $F_{i-1} \rightarrow F_i$ is given by a smash product of three factors
$$\cof(F_{i-1} \rightarrow F_i) \simeq X^G\sm \Sp^{\lfloor n/c_i \rfloor}(\Sigma^{\infty}S^0)\sm \Pri^{G/H_i}(S^{\infty\rreg_G})$$
where $c_i$ is the index of the subgroup $H_i$ in $G$. This is Proposition \ref{prop:stablefinitesymmetricpowers}. Our proof relies on Lemmas \ref{lemma:finitesymmetricpowersstablygenuine}, \ref{lemma:finitesymmetricpowersstablynaive}, and \ref{lemma:genuineandnaivefixedpoints} about finite symmetric powers in the stable setting, stated below and proven at the end of this section. These are equivariant analogues to several lemmas in (\cite{AD}, Section 7).

\begin{lemma}\label{lemma:finitesymmetricpowersstablygenuine}
The natural transformation
$$(-)\sm \Sp^n(\Sigma^{\infty G}S^0) \rightarrow \Sp^n(\Sigma^{\infty G}(-))$$
is an equivalence of functors from pointed $G$-spaces to genuine $G$-spectra. In particular, this functor preserves cofiber sequences.
\end{lemma}
\begin{lemma}\label{lemma:finitesymmetricpowersstablynaive}
The functor $\Sp^n(\Sigma^{\infty}(-))$ from pointed $G$-spaces to na{\"i}ve $G$-spectra preserves cofiber sequences.
\end{lemma}
\emph{Note:} It is NOT true that $X\sm \Sp^n(\Sigma^{\infty}S^0) \rightarrow \Sp^n(\Sigma^{\infty}X)$ is an equivalence! For example, if $X$ is a $G$-space whose fixed points are contractible, then the two na{\"i}ve $G$-spectra above have different $G$-fixed points.

\begin{lemma}\label{lemma:genuineandnaivefixedpoints}
For any nonnegative integer $\ell$, there is an inclusion of $G$-representation spheres $S^{\ell\rho_G} \hookrightarrow S^{\ell}\sm S^{\infty\rreg_G}$. The resulting natural transformation of functors from pointed $G$-spaces to spectra is an equivalence.
$$\Phi^G(\Sp^n(\Sigma^{\infty G}(-))) \rightarrow \Sp^n(\Sigma^{\infty}(S^{\infty\rreg_G}\sm (-)))^G$$
\end{lemma}

Given these three lemmas, we state and prove the main proposition of this section.

\begin{proposition}\label{prop:stablefinitesymmetricpowers}
Let $H_0, \ldots, H_m$ be a subgroup ordering of $G$, and for each $i$ let $c_i=[G:H_i]$. Then the functor $\Phi^G(\Sp^n(\Sigma^{\infty G}(-)))$ from pointed $G$-spaces to spectra has a filtration by functors
$$\xymatrix{\Sp^n(\Sigma^{\infty}(-)^G)=F_0\ar[r] & F_1\ar[r] & \cdots\ar[r] & F_m= \Phi^G(\Sp^n(\Sigma^{\infty G}(-)))}$$
where $F_{i-1} \rightarrow F_i$ is the homotopy fiber of a fibration
$$F_i \rightarrow (-)^G\sm \Sp^{\lfloor n/c_i\rfloor}(\Sigma^{\infty}S^0)\sm \Pri^{G/H_i}(S^{\infty\rreg_G}).$$
\end{proposition}
\begin{proof}
Let $X$ be any pointed $G$-space. By Lemma \ref{lemma:finitesymmetricpowersstablygenuine}, the natural inclusion
$$X^G\sm \Phi^G(\Sp^n(\Sigma^{\infty G}S^0)) \rightarrow \Phi^G(\Sp^n(\Sigma^{\infty G}X)$$
is an equivalence of spectra. It therefore suffices to prove the proposition when $X=S^0$. By Lemma \ref{lemma:genuineandnaivefixedpoints},
$$\Phi^G(\Sp^n(\Sigma^{\infty G}S^0))\simeq \Sp^n(\Sigma^{\infty}(S^{\infty\rreg_G}))^G$$
so we instead compute $\Sp^n(\Sigma^{\infty}(S^{\infty\rreg_G}))^G$. Consider the isotropy filtration of $S^{\infty\rreg_G}$.
$$\xymatrix{S^0\ar[r] & C_1\ar[r] & C_2\ar[r] & \cdots\ar[r] & C_m\ar[r] & S^{\infty\rreg_G}} \hskip 0.3in ; \hskip 0.3in C_i=\bigcup\limits_{H'\in [H_i]}S^{\infty\rreg_G^{H'}} $$
where the cofiber of the inclusion $C_{i-1} \rightarrow C_i$ is denoted $S^{\infty\rreg_G}(H_i)$. Applying the functor $\Sp^n(\Sigma^{\infty}(-))^G$ to this diagram, one obtains a filtration
$$\xymatrix{\Sp^n(\Sigma^{\infty}S^0)\ar[r] & \Sp^n(\Sigma^{\infty}C_1)^G\ar[r] & \cdots\ar[r] & \Sp^n(\Sigma^{\infty}C_m)^G\ar[r] & \Sp^n(\Sigma^{\infty}S^{\infty\rreg_G})^G}$$
By Proposition \ref{lemma:finitesymmetricpowersstablynaive}, the cofiber of the inclusion $\Sp^n(\Sigma^{\infty}C_{i-1})^G \rightarrow \Sp^n(\Sigma^{\infty}C_i)^G$ is
$$\Sp^n(\Sigma^{\infty}S^{\infty\rreg_G}(H_i))^G\simeq \Sp^{\lfloor n/c_i\rfloor}(\Sigma^{\infty}\Pri^{G/H_i}(S^{\infty\rreg_G}))$$
which completes the proof.
\end{proof}
Since the associated fibration sequences in the mod $p$ symmetric powers split (via Proposition \ref{prop:modpsymmetricpowersfixedpoints}), we deduce the following result in a manner similar to the proof of Proposition \ref{prop:stablefinitesymmetricpowers} above.
\begin{proposition}\label{prop:modpstablefinitesymmetricpowers}
Let $H_0, \ldots, H_m$ be a subgroup ordering of $G$, and for each $i$ let $c_i=[G:H_i]$. There is an equivalence between the two functors from pointed $G$-spaces to nonequivariant spectra,
$$\Phi^G\Sp^n_{\Z/p}(\Sigma^{\infty G}(-)) \simeq (-)^G \sm \bigvee\limits_{i=0}^{m}(\Sp^{\lfloor n/c_i\rfloor}_{\Z/p}(\Sigma^{\infty}S^0)\sm \Pri^{G/H_i}(S^{\infty\rreg_G})).$$
\end{proposition}

We now prove Lemmas \ref{lemma:finitesymmetricpowersstablygenuine}, \ref{lemma:finitesymmetricpowersstablynaive}, and \ref{lemma:genuineandnaivefixedpoints}.

\begin{lemma}\label{lemma:representationfixedpoints}
Let $\overline{\R^n}$ denote the reduced standard $(n-1)$-dimensional representation of $\Sigma_n$. Let $\Gamma$ be a subgroup of the Cartesian product $G\times \Sigma_n$ such that $\Gamma\cap (1\times \Sigma_n)$ is nontransitive. Then the $\Gamma$-fixed points of the representation $\rho_G\otimes \overline{\R^n}$ are nonzero.
\end{lemma}
\begin{proof}
Let $\Gamma'=\Gamma\cap (1\times \Sigma_n)$. Since $\Gamma'$ is nontransitive, there is a vector $w\in \overline{\R^n}$ which is fixed under $\Gamma'$. Pick any vector $v\in \rho_G$ such that the images $\{gv\}_{g\in G}$ are linearly indepdent, and consider the vector
$$u:= \sum\limits_{\gamma\in \Gamma}\gamma(v\otimes w)$$
Clearly $u$ is fixed by $\Gamma$, so it suffices to prove that $u$ is nonzero. Let $H$ denote the projection of $\Gamma$ onto $G$, and let $\Psi$ denote the projection of $\Gamma$ onto $\Sigma_n$. It is easily checked that $\Gamma'$ is a normal subgroup of $\Psi$, and so $\Gamma$ may be thought of as the graph of a homomorphism $f:H \rightarrow \Psi/\Gamma'$. Then
$$u=\sum\limits_{\gamma\in\Gamma}\gamma(v\otimes w)=\sum\limits_{h\in H}hv\otimes (|\Gamma'|\cdot f(h)w)$$
which is clearly nonzero because the images $\{hv\}_{h\in H}$ are linearly independent vectors in the representation $\rho_G$.
\end{proof}

\begin{proof}[Proof of Lemma \ref{lemma:finitesymmetricpowersstablygenuine}]
We use induction on $n$. The case $n=1$ is a tautology. Now consider general $n$. We will show that for any positive integer $\ell$ and any subgroup $H\subset G$, the natural map
$$(X\sm \Sp^n(S^{\ell\rho_G}))^H \rightarrow (\Sp^n(X\sm S^{\ell\rho_G}))^H$$
is an isomorphism on homotopy groups up through dimension $(2\ell-1)$. Note that for any pointed $G$-space $Y$, $\Sp^n(Y)/\Sp^{n-1}(Y)\simeq Y^{\sm n}/\Sigma_n$, and so it suffices for us to show that the natural map
$$X\sm (S^{\ell\rho_G})^{\sm n}/\Sigma_n \rightarrow (X\sm S^{\ell\rho_G})^{\sm n}/\Sigma_n$$
is $(2\ell-1)$-connected. To prove this, it suffices to show that for any subgroup $\Gamma\subset G\times \Sigma_n$, the map of fixed point spaces
$$(X\sm (S^{\ell\rho_G})^{\sm n})^{\Gamma} \rightarrow (X^{\sm n}\sm (S^{\ell\rho_G})^{\sm n})^{\Gamma}$$
is $(2\ell-1)$-connected. Let $H$ denote the projection of $\Gamma$ onto $G$. The two terms above can be rewritten as follows.
$$X^H\sm (S^{\ell\rho_G})^H\sm (S^{\ell(\rho_G\otimes \overline{\R^n})})^{\Gamma} \rightarrow (X^{\sm n})^{\Gamma}\sm (S^{\ell\rho_G})^H\sm (S^{\ell(\rho_G\otimes \overline{\R^n})})^{\Gamma}$$
If $\Gamma\cap (1\times \Sigma_n)$ is transitive, then $(X^{\sm n})^{\Gamma}=X^H$ and the map above is the identity map, so we are done. Suppose instead that $\Gamma\cap (1\times \Sigma_n)$ is nontransitive. Then $(S^{\ell\rho_G})^H$ has dimension at least $i$, and by Lemma \ref{lemma:representationfixedpoints}, $(S^{\ell(\rho_G\otimes \overline{\R^n})})^{\Gamma}$ has dimension at least $i$. Thus, the two spaces above are each at least a $2\ell$-fold suspension, and therefore the map shown is automatically $(2\ell-1)$-connected.
\end{proof}

\begin{proof}[Proof of Lemma \ref{lemma:finitesymmetricpowersstablynaive}]
We use induction on $n$. The base case $n=1$ is a tautology. Now consider general $n$. Let $X \rightarrow Y \rightarrow Z$ be a cofiber sequence of pointed $G$-spaces. By an argument analogous to the previous lemma, it suffices to show that for any subgroup $\Gamma\subset G\times \Sigma_n$, the map of fixed point spaces
$$((S^{\ell})^{\sm n}\sm Y^{\sm n}/X^{\sm n})^{\Gamma} \rightarrow ((S^{\ell})^{\sm n} \sm Z^{\sm n})^{\Gamma}$$
is $(2\ell-1)$-connected. Define $\Psi$ to be the projection of $\Gamma$ onto $\Sigma_n$. If $\Psi$ is nontransitive, then $(S^{\ell n})^{\Gamma}=S^{\ell\cdot o(\Psi)}$ where $o(\Psi)$ is the number of orbits of $\{1, \ldots, n\}$ under the action of $\Psi$. Then the two spaces above are each at least a $2\ell$-fold suspension, and the map above is automatically $(2\ell-1)$-connected.

Suppose that $\Psi$ is transitive. Let $H=\Gamma\cap (G\times 1)$. Then for any $G$-space $A$, the map
$$f:(A^{\sm n})^{\Gamma} \rightarrow A^H,\hskip 0.6in f(a_1, \ldots, a_n)=a_1$$
is a homeomorphism. This is true because $a_1$ can be freely chosen to be any $H$-fixed point of $A$, and each $a_i$ is uniquely determined by the point $a_1$. Therefore, $(Y^{\sm n}/X^{\sm n})^{\Gamma}\simeq Y^H/X^H$ and $(Z^{\sm n})^{\Gamma}\simeq Z^H$. The pointed spaces $Y^H/X^H$ and $Z^H$ are homotopy equivalent, and therefore the map of $\Gamma$-fixed points is an equivalence.
\end{proof}

\begin{proof}[Proof of Lemma \ref{lemma:genuineandnaivefixedpoints}]
By induction on $n$, it suffices to prove that the map of pointed spaces
$$((S^{\ell}\sm S^{\ell\rreg_G})^{\sm n})^{\Gamma} \rightarrow ((S^{\ell}\sm S^{\infty\rreg_G})^{\sm n})^{\Gamma}$$
is a $(2\ell-1)$-equivalence. If the $\Gamma$-fixed points of the representation $(\rreg_G\otimes \mathbb{R}^n)$ are zero, then the above map is an equivalence. If the $\Gamma$-fixed points of the representation $(\rreg_G\otimes \mathbb{R}^n)$ are nonzero, then both of the spaces above are $2\ell$-fold suspensions, and the map is a $(2\ell-1)$-equivalence.
\end{proof}


\subsection{Equivariant Eilenberg-Maclane Spectra}
\label{subsec:equivarianteilenbergmaclanespectra}

Let $X$ be a based $G$-CW complex and let $M$ be a $G$-module. Recall, either from Section \ref{subsec:symmetricpowerswithcoefficients} or from (\cite{DS}, Definition 2.1) that $M\otimes X$ denotes the free topological $M$-module over $X$ where $G$ acts diagonally. When $M=\Z$, one obtains a $G$-space weakly equivalent to the infinite symmetric power of $X$, and when $M=\Z/p$, one obtains the infinite mod $p$ symmetric power of $X$. These topological modules were studied in the papers \cite{LF} and \cite{DS}, where it was shown that the equivariant homotopy groups of $M\otimes X$ encode the equivariant homology groups of $X$.

\begin{definition}
Let $G$ be a finite group, and let $\underline{M}$ be a Mackey functor for the group $G$. Then there is a genuine $G$-spectrum $H\underline{M}$, called the \emph{Eilenberg-Maclane spectrum} for $\underline{M}$, which represents the functor for Bredon (co)homology with coefficients in $\underline{M}$.
\end{definition}

We will not use equivariant homotopy groups, equivariant homology groups, or Mackey functors in this paper, and so we will not define them --- the reader may consult \cite{MaG} for a reference. However, our main result is motivated by a desire to decompose the equivariant Eilenberg-Maclane spectrum $H\underline{\F}_p$, which is the genuine $G$-spectrum representing Bredon cohomology with coefficients in the trivial $G$-module $\F_p$. Thus, we illuminate the connection to mod $p$ symmetric powers.

The primary result we are interested in is (\cite{DS}, Theorem 1.1), which states that $M\otimes X$ is an equivariant infinite loop space and there is an equivalence natural in both the $G$-space $X$ and the $G$-module $M$,
$$\pi_V^G(M\otimes X) \cong \tilde{H}_V^G(X;\underline{M}).$$
When $G$ is the trivial group, the above result reduces to the classical Dold-Thom theorem. If we let $X$ denote the representation sphere $S^V$, then one obtains an equivalence of equivariant infinite loop spaces
$$M\otimes S^V \simeq K(\underline{M},V).$$
Using naturality, we obtain an equivalence of genuine $G$-spectra
$$M\otimes \Sigma^{\infty G}S^0 \simeq H\underline{M}.$$
If $M$ is a ring with a unit $1\in M$, then $H\underline{M}$ is a ring spectrum, and the inclusion $\Sigma^{\infty G}S^0 \hookrightarrow M\otimes \Sigma^{\infty G}S^0 \simeq H\underline{M}$ is the unit of the genuine $G$-spectrum $H\underline{M}$. The following corollary is immediate.
\begin{corollary}\label{cor:equivariantEilenbergMacLane}
There are filtrations of genuine $G$-spectra
$$\Sigma^{\infty G}S^0 \simeq \Sp^1(\Sigma^{\infty G}S^0) \subseteq \Sp^2(\Sigma^{\infty G}S^0) \subseteq \Sp^3(\Sigma^{\infty G}S^0) \subseteq \cdots\subseteq \Sp^{\infty}(\Sigma^{\infty G}S^0)\simeq H\underline{\Z}$$
$$\Sigma^{\infty G}S^0 \simeq \Sp_{\Z/p}^1(\Sigma^{\infty G}S^0) \subseteq \Sp_{\Z/p}^2(\Sigma^{\infty G}S^0) \subseteq \Sp_{\Z/p}^3(\Sigma^{\infty G}S^0) \subseteq \cdots\subseteq \Sp_{\Z/p}^{\infty}(\Sigma^{\infty G}S^0)\simeq H\underline{\F}_p.$$
The multiplication map on finite symmetric powers (resp. finite mod $p$ symmetric powers) in Definition \ref{definition:symmetricpowers} filters the ring structure of $H\underline{\Z}$ (resp. $H\underline{\F}_p$),
$$\Sp^m(\Sigma^{\infty G}S^0)\sm \Sp^n(\Sigma^{\infty G}S^0) \rightarrow \Sp^{mn}(\Sigma^{\infty G}S^0)$$
$$ \Sp_{\Z/p}^m(\Sigma^{\infty G}S^0)\sm \Sp_{\Z/p}^n(\Sigma^{\infty G}S^0) \rightarrow \Sp_{\Z/p}^{mn}(\Sigma^{\infty G}S^0).$$
\end{corollary}

Thus, our work in this section yields computations of the geometric fixed point spectra $\Phi^G(H\underline{\Z})$ and $\Phi^G(H\underline{\F}_p)$.
\begin{proposition}\label{prop:geometricfixedpointsofHZ}
Let $H_0, \ldots, H_m$ be a subgroup ordering of $G$. The geometric fixed point spectrum $\Phi^G(H\underline{\Z})$ has a filtration by $H\Z$-modules,
$$\xymatrix{H\Z=F_0\ar[r] & F_1\ar[r] & \cdots\ar[r] & F_m= \Phi^G(H\underline{\Z})}$$
where $F_{i-1} \rightarrow F_i$ is the homotopy fiber of a fibration of $H\Z$-modules
$$F_i \rightarrow H\Z\sm \Pri^{G/H_i}(S^{\infty\rreg_G}).$$
\end{proposition}
\begin{proof}
Apply Proposition \ref{prop:stablefinitesymmetricpowers}.
\end{proof}
\begin{proposition}\label{prop:geometricfixedpointsofHF_p}
Let $H_0, \ldots, H_m$ be a subgroup ordering of $G$. There is an equivalence of $H\F_p$-modules
$$\Phi^G(H\underline{\F}_p)\simeq \bigvee\limits_{i=0}^{m}(H\F_p\sm \Pri^{G/H_i}(S^{\infty\rreg_G})).$$
\end{proposition}
\begin{proof}
Apply Proposition \ref{prop:modpstablefinitesymmetricpowers}.
\end{proof}
In the next section, we compute an explicit expression for the space of primitives $\Pri^{G/H}(S^{\infty\rreg_G})$.

\section{The Primitives of an infinite representation sphere}

\label{sec:stableprimitives}
Let $G$ be a finite group, and let $X$ be a pointed $G$-space. We compute the primitives of the $G$-equivariant suspension spectrum of $X$, which is an equivalent task to describing the space $\Pri^{G/H}(S^{\infty\rreg_G}\sm X)$ (Proposition \ref{prop:stableprimitives}). To give the explicit description, we must give two definitions.
\begin{definition}\label{definition:diamond}
For any unpointed space $Y$, we use the notation $Y^{\Diamond}$ to denote the unreduced suspension of $Y$. The space $Y^{\Diamond}$ will be considered based, with the image of $Y\times \{0\} \hookrightarrow Y^{\Diamond}$ as the basepoint. In this way, even if $Y$ is not a based space, $Y^{\Diamond}$ is.
\end{definition}

\begin{definition}\label{definition:subgroupcomplex}
For any subgroup $H\subseteq G$, we write $\Po(G)_{\supset H}$ to be the poset of proper subgroups of $G$ which strictly contain $H$.
\end{definition}
Let $N_G(H):=\{g\in G:gHg^{-1}=H\}$ denote the normalizer subgroup of $H$, and let $W_G(H):=N_G(H)/H$ denote the Weyl group. Then Proposition \ref{prop:stableprimitives} states that
$$\Pri^{G/H}(S^{\infty\rreg_G}\sm X)\simeq \Sigma (\Po(G)_{\supset H}^{\Diamond})_{hW_G(H)}\sm X^G.$$
The main idea of the proof is that one may cover $\bigcup\limits_{K\supsetneq H}S^{\infty\rreg_G^K}$ by the contractible subspaces $S^{\infty\rreg_G^K}$ where $K$ varies over the poset $\Po(G)_{\supset H}\cup \{G\}$. With this covering one may prove, using Lemma \ref{lemma:coveringhomotopytype}, that $\bigcup\limits_{K\supsetneq H}S^{\infty\rreg_G^K}\simeq \Po(G)_{\supset H}^{\Diamond}$.

When we restrict our attention to $p$-groups $G$, a further simplification occurs. Any $p$-group $G$ has a maximal elementary abelian quotient $G \rightarrow G/F$, where $F$ denotes the kernel of the quotient map (Definition \ref{definition:Frattini}). If $H$ does contain the subgroup $F$, then we show the poset $\hat{\Po}(G)_{\supset H}$ is contractible (Lemma \ref{lemma:frattinilemma}). This allows us to simplify the expression above, and this simplification is stated in Proposition \ref{prop:stableprimitivespgroup}. Proposition \ref{prop:decompositionafterstableprimitives} combines all of the results of this section into a form we will use later.
\begin{definition}\label{definition:theposetC}
Let $G$ be a finite $p$-group. We let $\mathcal{C}$ be the set
$$\mathcal{C}:=\{H\unlhd G: G/H \text{ is an elementary abelian \emph{p}-group}\}.$$
The set $\mathcal{C}$ is closed under taking intersections, and its minimal element $F$ is the Frattini subgroup of $G$ (Definition \ref{definition:Frattini}). A subgroup $H$ is contained in $\mathcal{C}$ if and only if $F\subseteq H$.
\end{definition}
\begin{proposition}\label{prop:decompositionafterstableprimitives}
Let $G$ be a $p$-group. Then there is an equivalence of $\Z/p$-modules,
$$\Sp_{\Z/p}^{\infty}(S^{\infty\rreg_G})^G\simeq \prod\limits_{H\in\mathcal{C}}\Sp_{\Z/p}^{\infty}(\Sigma\Po(G)_{\supset H}^{\Diamond}\sm B(G/H)_+).$$
\end{proposition}
\begin{proof}
Proposition \ref{prop:modpsymmetricpowersfixedpoints} implies that there is a homeomorphism
$$\Sp_{\Z/p}^{\infty}(S^{\infty\rreg_G})^G\cong \prod\limits_{i=0}^{m}\Sp_{\Z/p}^{\infty}(\Pri^{G/H_i}(S^{\infty\rreg_G})),$$
where $H_0, \ldots, H_m$ are representatives for the conjugacy classes of subgroups of $G$. Proposition \ref{prop:stableprimitivespgroup} implies that
$$\Pri^{G/H}(S^{\infty\rreg_G})\simeq \begin{cases} \Sigma\Po(G)_{\supset H}^{\Diamond}\sm B(G/H)_+ \hskip 0.2in H\in\mathcal{C}\\
\star \hskip 1.7in H\notin\mathcal{C}\end{cases}.$$
The result immediately follows.
\end{proof}
\subsection{Coverings and Posets}

We familiarize the reader with a computational tool which will be crucial to us. Let $\Po$ be a finite poset. Suppose that $Y$ is a pointed topological space. A functor $\mathbf{Y}:\Po \rightarrow \Top_*$ is called a \emph{covering} of $Y$ if $\hocolim_{\Po}\mathbf{Y}\simeq Y$. If $\mathbf{Y}$ and $\mathbf{Y}'$ are two functors such that there is a natural homotopy equivalence between them, then $\hocolim_{\Po}\mathbf{Y}\simeq \hocolim_{\Po}\mathbf{Y}'$. If we replace $\mathbf{Y}$ by a diagram whose maps are all cofibrations, then the homotopy colimit is the union $\bigcup\limits_{a\in \Po}\mathbf{Y}(a)$.

Now suppose that for any two elements $a, b\in \Po$, there is a unique maximal element $a\sm b$ such that $a\ge a\sm b$ and $b\ge a\sm b$. Such posets are called \emph{meet semilattices.} If $\Po$ is a finite meet semilattice, then $\Po$ has a unique minimal element, which we denote by 0. Let $\Po_{>0}=\Po-\{0\}$.

\begin{lemma}\label{lemma:coveringhomotopytype}
Consider a functor $\mathbf{Y}:\Po \rightarrow \Top_*$ with the property that $\mathbf{Y}(a)\simeq \star$ whenever $a\neq 0$. Then $\hocolim_{\Po}\mathbf{Y}\simeq \mathbf{Y}(0)\sm \Po_{>0}^{\Diamond}$.
\end{lemma}
\begin{proof}
For any $a\in \Po$, let $\Po_{\backslash a}=\{x\in \Po:x\le a\}$. The pointed space $\mathbf{Y}(0)\sm \hat{\Po}^{\Diamond}$ has a covering $\mathbf{Y}'$ defined by
$$\mathbf{Y}'(a) = \mathbf{Y}(0)\sm (\Po_{\backslash a}-\{0\})$$
If $a\neq 0$, then the poset $\Po_{\backslash a}-\{0\}$ has a maximal element (namely, $a$), so it is contractible and thus $\mathbf{Y}'(a)\simeq \star$. If $a=0$, it's obvious that $\Po_{\backslash 0}-\{0\}$ is empty, and therefore its unreduced suspension is $S^0$. Thus, $\mathbf{Y}'\simeq \mathbf{Y}$, and so
$$\hocolim_{\Po}\mathbf{Y}\simeq \hocolim_{\Po}\mathbf{Y}'\simeq \mathbf{Y}(0)\sm (\Po_{>0})^{\Diamond}$$

\end{proof}
Now suppose that $\Po$ and $\mathcal{Q}$ are two posets and $f:\Po \rightarrow \mathcal{Q}$ is an order-preserving map. Let us suppose that $\mathbf{X}:\Po \rightarrow \Top_*$ and $\mathbf{Y}:\mathcal{Q} \rightarrow \Top_*$ are two functors such that
\begin{enumerate}
\item For every $a\in \Po$, there is a map $g_a:\mathbf{X}(a) \rightarrow \mathbf{Y}(f(a))$.
\item If $a\le b$, then the following diagram commutes
$$\xymatrix{\mathbf{X}(a)\ar[d] \ar[r]^{g_a} & \mathbf{Y}(f(a))\ar[d]\\
\mathbf{X}(b) \ar[r]^{g_b} & \mathbf{Y}(f(b))}$$\end{enumerate}
Then there is an induced map $\hocolim_{\Po}\mathbf{X} \rightarrow \hocolim_{\mathcal{Q}}\mathbf{Y}$. The lemma below states that the homotopy equivalence of Lemma \ref{lemma:coveringhomotopytype} is functorial in a strong sense.
\begin{lemma}
Suppose further that $\Po$ and $\mathcal{Q}$ are finite meet semilattices and for every $a, b\in \Po$, $f(a\sm b)=f(a)\sm f(b)$. If $\mathbf{X}(a)\simeq \star$ whenever $a\neq 0$ and $\mathbf{Y}(c)\simeq \star$ whenever $c\neq 0$, then the map $\hocolim_{\Po}\mathbf{X} \rightarrow \hocolim_{\mathcal{Q}}\mathbf{Y}$ is the map
$$g_0\sm f:\mathbf{X}(0)\sm (\Po_{>0})^{\Diamond} \rightarrow \mathbf{Y}(0)\sm (\mathcal{Q}_{>0})^{\Diamond}$$
\end{lemma}
\begin{proof}
As in the proof of the previous lemma, replace $\mathbf{X}$ with the homotopy equivalent functor $\mathbf{X}'$ defined by $\mathbf{X}'(a)=\mathbf{X}(0)\sm (\Po_{\backslash a}-\{0\})^{\Diamond}$, and replace $\mathbf{Y}$ with the homotopy equivalent functor $\mathbf{Y}'$ defined by $\mathbf{Y}'(c)=\mathbf{Y}(0)\sm (\mathcal{Q}_{\backslash c}-\{0\})^{\Diamond}$.
\end{proof}

\subsection{The Primitives of an infinite representation sphere}

\begin{proposition}\label{prop:stableprimitives}
Let $G$ be a finite group, and let $X$ be a pointed $G$-space. Then for any subgroup $H\subset G$,
$$\Pri^{G/H}(S^{\infty\rreg_G}\sm X)\simeq \Sigma (\Po(G)_{\supset H}^{\Diamond})_{hW}\sm X^G$$
where $W=W_G(H)$ is the Weyl group of $H$.
\end{proposition}

\begin{proof}
If $H=G$, then $\Pri^{G/G}(S^{\infty\rreg_G}\sm X)=S^0\sm X^G$, and by convention $\Sigma(\Po(G)_{\supset G}^{\Diamond})\sm X^G\simeq \Sigma(S^{-1})\sm X^G$. This deals with the case $H=G$, so let us henceforth assume $H\neq G$.

By definition, $\Pri^{G/H}(S^{\infty\rreg_G}\sm X)$ sits in a cofiber sequence
$$\xymatrix{W\backslash(\bigcup\limits_{K\supsetneq H}S^{\infty\rreg_G^K}\sm X^K)\ar[r]& W\backslash(S^{\infty\rreg_G^H}\sm X^H)\ar[r] &\Pri^{G/H}(S^{\infty\rreg_G}\sm X)}$$
Every point in the the quotient space $(S^{\infty\rreg_G^H}\sm X^H)/(\bigcup\limits_{K\supsetneq H}S^{\infty\rreg_G^K}\sm X^K)$ has isotropy group equal to $H$, and therefore $W$ acts freely on this quotient space. So we may replace the strict quotient by the homotopy quotient.
$$\xymatrix{(\bigcup\limits_{K\supsetneq H}S^{\infty\rreg_G^K}\sm X^K)_{hW}\ar[r]& (S^{\infty\rreg_G^H}\sm X^H)_{hW}\ar[r] &\Pri^{G/H}(S^{\infty\rreg_G}\sm X)}$$

Since $H\subsetneq G$, the $W$-space $S^{\infty\rreg_G^H}\sm X^H\sm EW_+$ has underlying points $S^{\infty}\sm X^H\sm EW_+\simeq \star$, and its fixed points under any nonzero subgroup of $W$ are also contractible because $W$ acts freely on $EW_+$. Therefore, $S^{\infty\rreg_G^H}\sm X^H\sm EW_+$ is contractible as a $W$-space, and so $(S^{\infty\rreg_G^H}\sm X^H)_{hW}\simeq \star$. It follows that $\Pri^{G/H}(S^{\infty\rreg_G}\sm X)\simeq \Sigma (\bigcup\limits_{K\supsetneq H}S^{\infty\rreg_G^K}\sm X^K)_{hW}$.
Let $\hat{\Po}(G)_{\supset H}$ denote the poset of all subgroups strictly containing $H$. The space $\bigcup\limits_{K\supsetneq H}S^{\infty\rreg_G^K}\sm X^K$ has a covering $\mathbf{Y}(-)$ indexed over $\hat{\Po}(G)_{\supset H}$. It is defined by
$$\mathbf{Y}(K)=S^{\infty\rreg_G^K}\sm X^K\simeq \begin{cases}
\star \hskip 0.4in K\neq G\\
X^G \hskip 0.2in K=G
\end{cases}$$
Therefore, Lemma \ref{lemma:coveringhomotopytype} implies that
$$\bigcup\limits_{K\supsetneq H}S^{\infty\rreg_G^K}\sm X^K=\hocolim_{\hat{\Po}_{\supset H}}\mathbf{Y}\simeq \mathbf{Y}(G)\sm \Po_{\supset H}^{\Diamond}\simeq X^G\sm \Po_{\supset H}^{\Diamond}$$
This completes the proof.

\end{proof}

\subsection{Subgroups complexes and the Frattini subgroup}

\begin{definition}\label{definition:Frattini}
Let $F\subset G$ denote the intersection of all of the maximal proper subgroups of $G$.  It is commonly referred to as the \emph{Frattini subgroup}. When $G$ is a $p$-group, $F$ is the minimal subgroup of $G$ such that $G/F$ is an elementary abelian $p$-group.
\end{definition}
\begin{lemma}\label{lemma:frattinilemma}
If $H$ does not contain the Frattini subgroup of $G$, then $\Po(G)_{\supset H}$ is contractible.
\end{lemma}
\begin{proof}
For any subgroup $K$ of $G$, let $KF$ denote the minimal subgroup of $G$ containing both $K$ and $F$. Consider the poset map
$$f:\Po(G)_{\supset H}\rightarrow \Po(G)_{\supset H}$$
$$K \mapsto KF$$
Any proper subgroup $K$ is contained in some maximal subgroup $M$ of $G$, and therefore the group $KF$ is also contained in $M=MF$. Therefore, $KF$ is also a proper subgroup of $G$ and so $f$ is a well defined map.

Since $KF\supset K$ for every $K$, the map $f$ is homotopic to the identity. Because $KF \supset HF$ for every $K$, the map $f$ is homotopic to the constant map at $HF$. It follows that $\Po(G)_{\supset H}$ is contractible.
\end{proof}

\begin{proposition}\label{prop:stableprimitivespgroup}
Suppose $G$ is a finite group and $X$ is any pointed $G$-space. If $H$ does not contain the Frattini subgroup of $G$, then $\Pri^{G/H}(S^{\infty\rreg_G}\sm X)\simeq \star$. If $H$ contains the Frattini subgroup of $G$, then
$$\Pri^{G/H}(S^{\infty\rreg_G}\sm X)\simeq \Sigma \Po(G)_{\supset H}^{\Diamond}\sm B(G/H)_+\sm X^G$$
\end{proposition}
\begin{proof}
By Proposition \ref{prop:stableprimitives},
$$\Pri^{G/H}(S^{\infty\rreg_G}\sm X)\simeq \Sigma (\Po(G)_{\supset H}^{\Diamond})_{hW}\sm X^G$$
If $H$ does not contain the Frattini subgroup, then $\Po(G)_{\supset H}^{\Diamond}$ has contractible underlying points and thus $(\Po(G)_{\supset H}^{\Diamond})_{hW}$ is contractible. If $H$ does contain the Frattini subgroup, then every subgroup containing $H$ is normal, and thus $W$ acts trivially on $\Po(G)_{\supset H}^{\Diamond}$. Thus $(\Po(G)_{\supset H}^{\Diamond})_{hW}\simeq \Po(G)_{\supset H}^{\Diamond}\sm BW_+$.
\end{proof}

\section{Products on Primitives}
\label{sec:productsonprimitives}
Let $X$ and $Y$ be pointed $G$-spaces. Consider the composition below, where the second map is inclusion of fixed points for the diagonal inclusion $G \subset G\times G$:
$$\Sp_{\Z/p}^{\infty}(X)^G \sm \Sp_{\Z/p}^{\infty}(Y)^G \rightarrow \Sp_{\Z/p}^{\infty}(X\sm Y)^{G\times G} \hookrightarrow \Sp_{\Z/p}^{\infty}(X\sm Y)^G.$$
Specialize to the case $X=Y=S^{\infty\rreg_G}$. There is a homeomorphism of $G$-spaces
$$X\sm Y = S^{\infty\rreg_G}\sm S^{\infty\rreg_G} \cong S^{\infty\rreg_G}$$
given by interleaving the copies of the reduced regular representation $\rreg_G$. Thus, we obtain a product structure on the $\Z/p$-module $\Sp_{\Z/p}^{\infty}(S^{\infty\rreg_G})^G$: it is a $\Z/p$-algebra whose 0 element is the basepoint.

Let us assume that $G$ is a $p$-group. The goal of this section is to analyze the product on $\Sp_{\Z/p}^{\infty}(S^{\infty\rreg_G})^G$ in terms of its free $\Z/p$-module basis which arises from Propositions \ref{prop:modpsymmetricpowersfixedpoints} and \ref{prop:decompositionafterstableprimitives}:
\begin{equation}\label{eqn:pgpfixptssymmpwrrepsph}
\Sp_{\Z/p}^{\infty}(S^{\infty\rreg_G})^G\simeq \prod\limits_{H\in\mathcal{C}}\Sp_{\Z/p}^{\infty}(\Pri^{G/H}(S^{\infty\rreg_G}))\simeq \prod\limits_{H\in\mathcal{C}}\Sp_{\Z/p}^{\infty}(\Sigma\Po(G)_{\supset H}^{\Diamond}\sm B(G/H)_+).
\end{equation}
Let $H, K\in\mathcal{C}$ (Definition \ref{definition:theposetC}). Then the $H$-factor smashed with the $K$-factor maps to the $(H\cap K)$-factor in the decomposition of Equation \ref{eqn:pgpfixptssymmpwrrepsph}, in the following way. Let $p^m$ denote the ratio $\frac{[G:H][G:K]}{[G:(H\cap K)]}$, where $m\ge 0$.  The product on the basis elements $\Pri^{G/H}(S^{\infty\rreg_G})$ is the composition
\begin{equation}\label{eqn:productonprimitives}
\Pri^{G/H}(S^{\infty\rreg_G}) \sm \Pri^{G/K}(S^{\infty\rreg_G}) \rightarrow \Pri^{\frac{G\times G}{H\times K}}(S^{\infty\rreg_G}\sm S^{\infty\rreg_G}) \rightarrow \Sp_{\Z/p}^{p^m}(\Pri^{G/(H\cap K)}(S^{\infty\rreg_G})).
\end{equation}
The first map in Equation \ref{eqn:productonprimitives} is the `product' for primitives (Definition \ref{definition:primitives}) and the second is `restriction' for primitives (Definition \ref{definition:primitivesrestriction}) along the diagonal inclusion $G\subset G\times G$.
\begin{definition}\label{definition:transverse}
Suppose that $G$ is a $p$-group, and let $H, K\in \mathcal{C}$ (Definition \ref{definition:theposetC}). We say that $H$ and $K$ are \emph{transverse} if the diagonal inclusion below is an isomorphism,
$$G/(H\cap K) \hookrightarrow G/H\times G/K.$$
Equivalently, $H$ and $K$ are transverse iff $\frac{[G:H][G:K]}{[G:(H\cap K)]}=1$. If the diagonal inclusion is not an isomorphism, we say $H$ and $K$ are \emph{nontransverse}.
\end{definition}

In the case where $H$ and $K$ are transverse, the composition of Equation \ref{eqn:productonprimitives} can be described explicitly in terms of a natural product on subgroup complexes, and this is done in Proposition \ref{prop:transversecase}. If $H$ and $K$ are nontransverse the situation is more subtle, and for our purposes the following weaker result suffices. Consider the stable analogue of Equation \ref{eqn:pgpfixptssymmpwrrepsph}, which results from Propositions \ref{prop:modpstablefinitesymmetricpowers} and \ref{prop:decompositionafterstableprimitives}:
\begin{equation}
\label{eqn:pgpfixptssymmpwrrepsphstable}
\Phi^G(\Sp_{\Z/p}^{\infty}(\Sigma^{\infty G}S^0))\simeq \bigvee\limits_{H\in\mathcal{C}}\Sp_{\Z/p}^{\infty}(\Sigma^{\infty}S^0)\sm \Pri^{G/H}(S^{\infty\rreg_G}).
\end{equation}
The stable version of Equation \ref{eqn:productonprimitives} (with the middle term omitted) is
\begin{equation}\label{eqn:productonprimitivesstable}
\Sigma^{\infty}\Pri^{G/H}(S^{\infty\rreg_G})\sm \Sigma^{\infty}\Pri^{G/K}(S^{\infty\rreg_G}) \rightarrow \Sp_{\Z/p}^{p^m}(\Sigma^{\infty}S^0)\sm \Sigma^{\infty}\Pri^{G/(H\cap K)}(S^{\infty\rreg_G})
\end{equation}
\begin{definition}\label{definition:reducedsymmetricpowers}
Let $\overline{\Sp}_{\Z/p}^n(-)$ denote the functor
$$\overline{\Sp}_{\Z/p}^n(-):=\Sp_{\Z/p}^n(-)/\Sp_{\Z/p}^{n-1}(-).$$
There is an obvious quotient map $\Sp_{\Z/p}^n(-) \rightarrow \overline{\Sp}_{\Z/p}^n(-)$.
\end{definition}
Proposition \ref{prop:nontransversecase} says that the following composition is zero on $\F_p$-homology:
$$\xymatrix{\Sigma^{\infty}\Pri^{G/H}(S^{\infty\rreg_G})\sm \Sigma^{\infty}\Pri^{G/K}(S^{\infty\rreg_G})\ar@{-->}[dr] \ar[r] &  \Sp_{\Z/p}^{p^m}(\Sigma^{\infty}S^0)\sm \Sigma^{\infty}\Pri^{G/(H\cap K)}(S^{\infty\rreg_G})\ar[d]\\
& \overline{\Sp}_{\Z/p}^{p^m}(\Sigma^{\infty}S^0)\sm \Sigma^{\infty}\Pri^{G/(H\cap K)}(S^{\infty\rreg_G})}.$$
This requires some work to do, and is a result of the following more general lemma. Let $V$ be an $\F_p$-vector space of rank $m$, viewed as an abelian group. Let $X$ be a pointed $V$-space, and let $\ell \gg 0$ be a positive integer,. Then we prove (Lemma \ref{lemma:zeroonhomology}) that the composition below, where the first map is the restriction and the second is the quotient
$$\Pri^V(\Sigma^{\ell}X) \rightarrow \Sp_{\Z/p}^{p^m}(\Sigma^{\ell}X) \rightarrow \overline{\Sp}_{\Z/p}^{p^m}(\Sigma^{\ell}X)$$
is zero on $\F_p$-homology in degrees $\ell, \ell+1, \ldots, p\ell-1$. We then take $X=\Pri^{G/(H\cap K)}(S^{\infty\rreg_G}\sm S^{\infty\rreg_G})$ and $V=\frac{G/H\times G/K}{G/(H\cap K)}$, and let $\ell\to\infty$ to deduce Proposition \ref{prop:nontransversecase}.

\subsection{Subgroup Complexes and Transverse Subgroups}
\label{subsec:subgroupcomplexesandtransversesubgroups}
Recall that $\Po(G)_{\supset H}$ denotes the poset of proper subgroups of $G$ which strictly contain $H$ (Definition \ref{definition:subgroupcomplex}), and $\Po(G)_{\supset H}^{\Diamond}$ denotes its unreduced suspension, regarded as a pointed space. We write $\hat{\Po}(G)_{\supseteq H}$ to mean the poset of all subgroups of $G$ which contain $H$ (including both $H$ and $G$ itself). Then there is a map of posets
$$\hat{\Po}(G)_{\supseteq H}\times \hat{\Po}(G)_{\supseteq K} \rightarrow \hat{\Po}(G)_{\supseteq H\cap K}$$
$$(H', K') \mapsto H'\cap K'$$
If $H$ and $K$ are transverse, then $H'\cap K'=H\cap K$ if and only if $H'=H$ and $K'=K$. Therefore, we obtain a product map on the unreduced join of the two spaces $\Po(G)_{\supset H}^{\Diamond}$ and $\Po(G)_{\supset K}^{\Diamond}$.
$$\Po(G)_{\supset H}^{\Diamond} \star\Po(G)_{\supset K}^{\Diamond} \rightarrow \Po(G)_{\supset H\cap K}^{\Diamond}.$$
Note that the join of two spaces is the same as the suspension of their smash product.
\begin{definition}\label{definition:subgroupcomplexproduct}
Let $G$ be a group, and let $H, K\unlhd G$ be subgroups such that the diagonal inclusion $G/(H\cap K) \rightarrow G/H\times G/K$ is an isomorphism. The assignment $(H',K') \mapsto H'\cap K'$ yields a map of pointed spaces
$$\Sigma \Po(G)_{\supset H}^{\Diamond}\sm \Sigma \Po(G)_{\supset K}^{\Diamond} \rightarrow \Sigma \Po(G)_{\supset H\cap K}^{\Diamond},$$
which we call the \emph{subgroup complex product.}
\end{definition}
\begin{proposition}\label{prop:transversecase}
Let $G$ be a $p$-group. Let $H, K\in\mathcal{C}$ be transverse subgroups of $G$. Under the equivalence
$$\Pri^{G/H}(S^{\infty\rreg_G})\simeq \Sigma \Po(G)^{\Diamond}_{\supset H}\sm B(G/H)_+$$
of Proposition \ref{prop:stableprimitivespgroup}, the composite map
$$\Pri^{G/H}(S^{\infty\rreg_G})\sm \Pri^{G/K}(S^{\infty\rreg_G}) \rightarrow \Pri^{G/(H\cap K)}(S^{\infty\rreg_G})$$
in Equation \ref{eqn:productonprimitives} is give by the subgroup complex product, smashed with the equivalence
$$B(G/H)_+\sm B(G/K)_+ \rightarrow B(G/(H\cap K))_+$$
\end{proposition}
\begin{proof}

The fixed point space $S^{\infty\rreg_G^H}$ has a covering indexed over the poset $\Po(G)_{\supseteq H}$. The cover element corresponding to the subgroup $H'\supseteq H$ is $S^{\infty\rreg_G^{H'}}$. The fixed point space $S^{\infty\rreg_G^K}$ has a similar covering. The behavior of the interleaving map
$$S^{\infty\rreg_G^H} \sm S^{\infty\rreg_G^K} \subset S^{\infty\rreg_G^{H\cap K}}\sm S^{\infty\rreg_G^{H\cap K}}\simeq S^{\infty\rreg_G^{H\cap K}}$$
with respect to this covering sends $(H', K')$ to $H'\cap K'$. Because $H$ and $K$ are transverse, one has that $H'\cap K'=H\cap K$ if and only if $H'=H$ and $K'=K$. Therefore, the interleaving map descends to the associated graded
$$(S^{\infty\rreg_G^H}/\bigcup\limits_{H'\supsetneq H}S^{\infty\rreg_G^{H'}}) \sm (S^{\infty\rreg_G^K}/\bigcup\limits_{K'\supsetneq K}S^{\infty\rreg_G^{K'}}) \rightarrow (S^{\infty\rreg_G^{H\cap K}}/\bigcup\limits_{L\supsetneq H\cap K}S^{\infty\rreg_G^L})$$
and this map is given by the subgroup complex product
$$\Sigma\Po(G)_{\supset H}^{\Diamond}\sm \Sigma\Po(G)_{\supset K}^{\Diamond} \rightarrow \Sigma\Po(G)_{\supset H\cap K}^{\Diamond}.$$
The action of $G/(H\cap K)$ on each side is trivial, so upon applying $G/(H\cap K)$ homotopy orbits, we get the desired result.

\end{proof}
\subsection{Cyclic Powers and Primitives}
\label{subsec:cyclicpowersandprimitives}
Let $Y$ be a pointed space. The symmetric powers of $Y$ are defined in terms of the symmetric group. We will define \emph{mod $p$ cyclic powers} in terms of elementary abelian $p$-groups. The mod $p$ cyclic powers of $Y$ will be used to describe the restriction map on primitives.

\begin{definition}\label{definition:cyclicpowers}
Let $m$ be any positive integer, and let $V\cong (\Z/p)^m$. Let $\fun(V,Y)$ be the space of maps from $V$ (viewed as a discrete space) to $Y$, with an action of $V$ given by
$$(vf)(w)=f(v^{-1}w) \hskip 0.3in v,w\in V, f\in \fun(V,Y).$$
For two functions $f, f'\in \fun(V,Y)$, suppose that there are $p$ points $v_1, \ldots, v_p\in V$ such that
$$f(v_1)=\cdots=f(v_p) , \hskip 0.2in f'(v_1)=\cdots=f(v_p) , \hskip 0.2in \mathrm{and}$$
$$f(v)=f'(v) \hskip 0.3in \forall v\in V-\{v_1, \ldots, v_p\}.$$
Then we write $f\sim f'$, and let $\sim$ denote the equivalence relation on $\fun(V,Y)$ generated by the above conditions. The quotient $\frac{\fun(V,Y)/\sim}{V}$ is the \emph{$V$-th mod $p$ cyclic power} of $Y$, and is denoted by $\cyc_{\Z/p}^V(Y)$. There is an obvious quotient map $\cyc_{\Z/p}^V(Y) \rightarrow \Sp_{\Z/p}^{p^m}(Y).$
\end{definition}
{\bf Note:} The equivalence relation $\sim$ is what makes it `mod $p$'. Without this relation, we obtain a construction $\fun(V,Y)/V$ which we would call $\cyc^V(Y)$ if we had any use for it.
\begin{definition}
Let ${\fun}_0(V,Y)\subset \fun(V,Y)$ denote the subspace consisting of those maps which send at least one vector $v$ to the basepoint of $Y$, and let $\overline{\fun}(V,Y):=\fun(V,Y)/\fun_0(V,Y)$. The quotient $\frac{\overline{\fun}(V,Y)/\sim}{V}$ is called the \emph{reduced $V$-th mod $p$ cyclic power of $Y$}, and is denoted by $\overline{\cyc}_{\Z/p}^V(Y)$. There is a commutative square of
$$\xymatrix{\cyc_{\Z/p}^V(Y)\ar[r]\ar[d] & \Sp_{\Z/p}^{p^m}(Y)\ar[d]\\
\overline{\cyc}_{\Z/p}^V(Y)\ar[r] & \overline{\Sp}_{\Z/p}^{p^m}(Y)}.
$$
\end{definition}

Now suppose that $Y$ is a pointed $V$-space, and let $Y^{\{e\}}$ denote its underlying points. Let $\varphi:Y \rightarrow \fun(V,Y^{\{e\}})$ denote the $V$-equivariant map defined by $(\varphi(y))(v)=vy$. Let $0\in \fun(V, Y^{\{e\}})$ denote the function which sends every element of $V$ to the basepoint of $Y^{\{e\}}$. It is easily seen that if $W\subset V$ is a nonzero subspace and $y \in Y^W$, then $\varphi(y) \sim 0$. Therefore, $\varphi$ determines a map (of pointed spaces)
\begin{equation}\label{eqn:primitivestocyclicpowers}
\overline{\varphi}: \Pri^V(Y):=\frac{Y/\bigcup\limits_{W\neq 0}Y^W}{V} \rightarrow \frac{\fun(V,Y^{\{e\}})/\sim}{V}=:\cyc_{\Z/p}^V(Y^{\{e\}})
\end{equation}
Crucially for us, we have a commutative diagram where the composition along the top row is the restriction map of Definition \ref{definition:primitivesrestriction}:
\begin{equation}\label{eqn:cyclictosymmetric}
\xymatrix{\Pri^V(Y)\ar[r]^{\overline{\varphi}}\ar@{-->}[dr] & \cyc_{\Z/p}^V(Y^{\{e\}})\ar[r]\ar[d] & \Sp_{\Z/p}^{p^m}(Y^{\{e\}})\ar[d]\\
& \overline{\cyc}_{\Z/p}^V(Y^{\{e\}})\ar[r] & \overline{\Sp}_{\Z/p}^{p^m}(Y^{\{e\}})}.
\end{equation}

The entire preceding discussion is natural in $Y$, i.e. all maps arise from natural transformations of functors from pointed $V$-spaces to pointed spaces. The following lemma describes the behavior of the dotted map in the stable range.

\begin{lemma}\label{lemma:zeroonhomology}
Let $V\cong (\Z/p)^m$, let $X$ be a pointed space with an action of $V$, and let $\ell\gg 0$ be a positive integer. Then letting $Y=\Sigma^{\ell}X$ in Equation \ref{eqn:cyclictosymmetric}, the composition
$$\Pri^V(\Sigma^{\ell}X) \rightarrow \cyc_{\Z/p}^V(\Sigma^{\ell}X^{\{e\}}) \rightarrow \overline{\cyc}_{\Z/p}^V(\Sigma^{\ell}X^{\{e\}})$$
is zero on reduced $\F_p$-homology in degrees $\ell, \ell+1, \ldots, p\ell-1$. Therefore, the map of spectra
$$\Pri^V(\Sigma^{\infty}X) \rightarrow \cyc_{\Z/p}^V(\Sigma^{\infty}X^{\{e\}}) \rightarrow\overline{\cyc}_{\Z/p}^V(\Sigma^{\infty}X)$$
is zero on $\F_p$-homology. And therefore, the composition of the restriction map (Definition \ref{definition:primitivesrestriction}) with the quotient map is zero on $\F_p$-homology:
$$\Pri^V(\Sigma^{\infty}X) \rightarrow \Sp_{\Z/p}^{p^m}(\Sigma^{\infty}X^{\{e\}}) \rightarrow\overline{\Sp}_{\Z/p}^{p^m}(\Sigma^{\infty}X).$$
\end{lemma}
\begin{proof}
We first address the case where $\dim(V)=1$, i.e. $V\cong \Z/p$. Let $v\in V$ denote a generator. Write $Y=\Sigma^{\ell}X$. Then the $V$-space $\fun(V,Y^{\{e\}})$ can be identified with the Cartesian product $Y^{\times p}$ with the permutation action of the group $\Z/p$.

Let $\Delta:Y \rightarrow Y^{\times p}$ denote the diagonal map, and let $\Delta_{\mathrm{tw}}:Y \rightarrow Y^{\times p}$ denote the twisted diagonal map $\Delta_{\mathrm{tw}}(y)=(y, vy, v^2y, \ldots, v^{p-1}y)$. Consider the following commutative diagram of pointed spaces, where the maps $Y^V \rightarrow Y$ are the inclusion of the fixed points, and the columns are cofiber sequences.
\begin{equation}\label{eqn:zeroonhomologydiagram1}
\xymatrix{Y^V\ar[d]\ar[r] & Y\ar@{=}[r]\ar[d]^{\Delta} & Y\ar[d]^{\Delta}\\
Y\ar[r]^{\Delta_{\mathrm{tw}}}\ar[d] & Y^{\times p}\ar[r]\ar[d] & Y^{\sm p}\ar[d]\\
Y/Y^V\ar[r] & Y^{\times p}/\im(\Delta)\ar[r] & Y^{\sm p}/\im(\Delta)}
\end{equation}
After quotienting by the free action of $V$, the bottom row of Diagram \ref{eqn:zeroonhomologydiagram1} is the composition
$$\Pri^V(Y) \rightarrow \cyc_{\Z/p}^V(Y^{\{e\}})\rightarrow \overline{\cyc}_{\Z/p}^V(Y^{\{e\}}).$$
Therefore, it suffices to prove that the composition along the bottom row of Diagram \ref{eqn:zeroonhomologydiagram1} is zero on homology in the stable range.

The space $Y^{\sm p}$ has homology concentrated in degree $p\ell$ and higher, and so it is stably contractible. Apply the functor $\tilde{H}_*(-)$ to Diagram \ref{eqn:zeroonhomologydiagram1} for any $\ell\le *\le p\ell-1$, omitting the middle column and adding another row at the top to obtain
\begin{equation}\label{eqn:zeroonhomologydiagram2}
\xymatrix{\tilde{H}_{*+1}(Y/Y^V)\ar[r]^{h}\ar[d]_{j}& \tilde{H}_{*+1}(Y^{\sm p}/\im(\Delta))\ar@{=}[d]\\
\tilde{H}_*(Y^V)\ar[r]^{f}\ar[d]_{g} & \tilde{H}_*(Y)\ar[d]\\
\tilde{H}_*(Y)\ar[r]\ar[d] & 0\ar[d]\\
\tilde{H}_*(Y/Y^V)\ar[r]^{\Sigma h}& \tilde{H}_*(Y^{\sm p}/\im(\Delta)).}
\end{equation}
The maps $f$ and $g$ are equal, and $g\circ j=0$. Thus, $f\circ j=0$, and thus $h=0$. Thus, $\Sigma h=0$, as desired. 

Now suppose $\dim(V)\ge 2$. Write $V\cong L\oplus W$ where $\dim(L)=1$. There is an obvious natural transformation $\fun(L, \fun(W, -)) \rightarrow \fun(V, -)$ , and therefore a natural transformation
$$\cyc_{\Z/p}^L(\cyc_{\Z/p}^W(-)) \rightarrow \cyc_{\Z/p}^V(-).$$ This natural transformation descends to the reduced cyclic powers $\overline{\cyc}_{\Z/p}^L(\overline{\cyc}_{\Z/p}^W(-)) \rightarrow \overline{\cyc}_{\Z/p}^V(-)$. The composition
$$\Pri^V(-)\cong \Pri^L(\Pri^W(-)) \rightarrow \overline{\cyc}_{\Z/p}^L(\Pri^W(-))\rightarrow \overline{\cyc}_{\Z/p}^L(\overline{\cyc}_{\Z/p}^W(-)) \rightarrow \overline{\cyc}_{\Z/p}^V(-)$$
is the restriction $\Pri^V(-) \rightarrow \overline{\cyc}_{\Z/p}^V(-)$. By the dimension 1 case, the first step of this composition
$$\Pri^L(\Pri^W(\Sigma^{\ell}X))=\Pri^L(\Sigma^{\ell}\Pri^W(X)) \rightarrow \overline{\cyc}_{\Z/p}^L(\Sigma^{\ell}\Pri^W(X))=\overline{\cyc}_{\Z/p}^L(\Pri^W(\Sigma^{\ell}X))$$
is zero on $\F_p$-homology in degrees $\ell, \ell+1, \ldots, p\ell-1$. Therefore, the composition $\Pri^V(\Sigma^{\ell}X) \rightarrow \overline{\cyc}_{\Z/p}^V(\Sigma^{\ell}X)$ is zero on $\F_p$-homology in degrees $\ell, \ell+1, \ldots, p\ell-1$.

\end{proof}

\begin{proposition}\label{prop:nontransversecase}
Let $G$ be a $p$-group. Let $H, K\in\mathcal{C}$ be nontransverse subgroups of $G$. Let $p^m=\frac{[G:H][G:K]}{G:(H\cap K)]}$. Consider the following diagram where the horizontal map is the product of Diagram \ref{eqn:productonprimitivesstable} and the vertical map is the quotient
$$\xymatrix{\Sigma^{\infty}\Pri^{G/H}(S^{\infty\rreg_G})\sm \Sigma^{\infty}\Pri^{G/K}(S^{\infty\rreg_G})\ar@{-->}[dr] \ar[r] &  \Sp_{\Z/p}^{p^m}(\Sigma^{\infty}S^0)\sm \Sigma^{\infty}\Pri^{G/(H\cap K)}(S^{\infty\rreg_G})\ar[d]\\
& \overline{\Sp}_{\Z/p}^{p^m}(\Sigma^{\infty}S^0)\sm \Sigma^{\infty}\Pri^{G/(H\cap K)}(S^{\infty\rreg_G})}.$$
The composition (dotted) is zero on $\F_p$-homology.
\end{proposition}
\begin{proof}
Take $X=\Pri^{G/(H\cap K)}(S^{\infty\rreg_G}\sm S^{\infty\rreg_G})$, where $G$ acts diagonally on $S^{\infty\rreg_G}\sm S^{\infty\rreg_G}$. Let $V=\frac{G/H\times G/K}{G/(H\cap K)}$. Observe that $X$ has a residual action of $V$, and
$$\Pri^V(\Sigma^{\infty}X) \simeq \Sigma^{\infty}\Pri^{G/H}(S^{\infty\rreg_G})\sm \Sigma^{\infty}\Pri^{G/K}(S^{\infty\rreg_G}).$$
The last line of Lemma \ref{lemma:zeroonhomology} immediately implies the desired result.
\end{proof}

\section{The first layer of the filtration}
\label{sec:firstcofiber}
Let $G$ be a $p$-group, and let $X$ be any pointed $G$-space. Let $\Aff_1\simeq (\Z/p)\rtimes \GL_1(\F_p)$ be the group consisting of affine transformation of the one-dimensional vector space over $\F_p$, and fix an inclusion $\Aff_1\subseteq \Sigma_p$ into the symmetric group on $p$ letters. The equivariant classifying space $B_G(-)$ is a functor from groups to $G$-space, and is constructed in Definition \ref{definition:equivariantclassifyingspace}. The notion of a $p$-local equivalence of spectra is Definition \ref{definition:plocal}. The main result of this section is the following proposition, which is proven here using several intermediate results to be proven in the body of this section.

\begin{proposition}\label{prop:firstcofiber}
There is a $p$-local equivalence of genuine $G$-spectra 
$$\Sp_{\Z/p}^p(\Sigma^{\infty G}X)/\Sp^1(\Sigma^{\infty G}X) \simeq X\sm S^1\sm \Sigma_+^{\infty G}B_G\Aff_1.$$
\end{proposition}
\emph{Note:} The equivalence above is the $n=1$ case of Theorem \ref{thm:maintheorem}.
\begin{proof}
There is a cofiber sequence of functors
$$\xymatrix{\Sp^1(-)\ar[r]^{\Delta} & \Sp^p(-)\ar[r] & \Sp_{\Z/p}^p(-)}$$
where $\Delta$ is the diagonal map. Thus by Lemma \ref{lemma:finitesymmetricpowersstablygenuine}, there is a natural equivalence of genuine $G$-spectra
$$\Sp_{\Z/p}^p(\Sigma^{\infty G}X)/\Sp_{\Z/p}^1(\Sigma^{\infty G}X)\simeq X\sm \Sp_{\Z/p}^p(\Sigma^{\infty G}S^0)/\Sp_{\Z/p}^1(\Sigma^{\infty G}S^0).$$
Proposition \ref{prop:earlierlayers} implies that the quotient map is a $p$-local equivalence
$$\Sp_{\Z/p}^p(\Sigma^{\infty G}S^0)/\Sp^1(\Sigma^{\infty G}S^0)\simeq \Sp_{\Z/p}^p(\Sigma^{\infty G}S^0)/\Sp^{p-1}(\Sigma^{\infty G}S^0).$$
Let $\mathcal{F}$ denote the family of nontransitive subgroups of $\Sigma_p$. Proposition \ref{prop:geometricargument} says that
$$\Sp_{\Z/p}^p(\Sigma^{\infty G}S^0)/\Sp^{p-1}(\Sigma^{\infty G}S^0)\simeq \Sigma^{\infty G}(S^1\sm E_G\mathcal{F}_+/\Sigma_p).$$
Propositions \ref{prop:FandSigmap} and \ref{prop:affsigma} say that there are $p$-local equivalences of equivariant classifying spaces
$$E_G\mathcal{F}_+/\Sigma_p \simeq (B_G\Sigma_p)_+ \simeq (B_G\Aff_1)_+.$$
\end{proof}

\subsection{Equivariant Classifying Spaces}

The following definition of a $G$-equivariant classifying space is given in \cite{Sankar}.

\begin{definition}\label{definition:equivariantclassifyingspace}
Let $\Lambda$ be any finite group. Suppose that $\mathcal{F}$ is a collection of subgroups of $\Lambda$ with the property that if $\Gamma \in \mathcal{F}$, then every subgroup of $\Gamma$ and every group conjugate to $\Gamma$ is in $\mathcal{F}$. Then we define $E_G\mathcal{F}$ to be the $(G\times\Lambda)$-space with fixed points under any subgroup $\Gamma\subset (G\times \Lambda)$
$$(E_G\mathcal{F})^{\Gamma} \simeq \begin{cases}
* \hskip 0.1in \mathrm{if} \hskip 0.1in \Gamma\cap \Lambda \in \mathcal{F}\\
\emptyset \hskip 0.1in \mathrm{if} \hskip 0.1in \Gamma\cap\Lambda \notin \mathcal{F}
\end{cases}$$
When $\mathcal{F}$ contains only the trivial group, $E_G\mathcal{F}$ is denoted $E_G\Lambda$. This space has a free action of $\Lambda$. We call $B_G\Lambda=(E_G\Lambda)/\Lambda$ the $G$-\emph{equivariant classifying space} of $\Lambda$.
\end{definition}

\begin{example}\label{example:equivariantlensspace}
Let
\begin{itemize}
\item $\rho_G$ be the real regular representation of $G$,
\item $\overline{\R^n}$ be the reduced permutation representation of $\Sigma_n$,
\item $\mathcal{F}$ denote the family of nontransitive subgroups of $\Sigma_n$, and
\item for any representation $V$, let $U(V)$ denote the unit sphere of $V$.
\end{itemize}
Then Lemma \ref{lemma:representationfixedpoints} implies that $U(\infty\rho_G\otimes \overline{\R^n})=E_G\mathcal{F}$.
\end{example}

Specialize the above example to the case $n=p$.

\begin{proposition}\label{prop:geometricargument}
Let $\mathcal{F}$ denote the collection of nontransitive subgroups of $\Sigma_p$. The cofiber $\Sp_{\Z/p}^p(\Sigma^{\infty G}S^0)/\Sp_{\Z/p}^{p-1}(\Sigma^{\infty G}S^0)$ is equivalent to to $\Sigma^{\infty G}(S^1 \sm E_G\mathcal{F}_+/\Sigma_p)$.
\end{proposition}
\begin{proof}
Let $X=S^V$, for any $G$-representation $V$. Then
$$\begin{aligned}
\Sp_{\Z/p}^p(S^V)/\Sp_{\Z/p}^{p-1}(S^V) &\simeq \cof(\xymatrix{S^V\ar[r]^{\Delta\hskip 0.2in} & (S^V)^{\sm p}/\Sigma_p})\\
&\simeq \cof(\xymatrix{S^V\ar[r]^{\Delta\hskip 0.2in} & (S^V)^{\sm p}})/\Sigma_p\\
&\simeq S^V\sm \cof(\xymatrix{S^0\ar[r]& S^{V\otimes\overline{\R^p}}})/\Sigma_p\\
&\simeq S^V\sm S^1\sm U(V\otimes \overline{\R^p})_+/\Sigma_p.
\end{aligned}$$
It therefore suffices for us to prove that $U(\infty(\rho_G\otimes \overline{\R^p}))\simeq E_G\mathcal{F}$. It is an easy consequence of Lemma \ref{lemma:representationfixedpoints} that for any subgroup $\Gamma\subset G\times \Sigma_p$,
$$U(\infty(\rho_G\otimes \overline{\R^p}))=\begin{cases}U(0) = \emptyset \hskip 0.15in \mathrm{if} \hskip 0.15in \Gamma\cap \Sigma_p \hskip 0.15in \mathrm{transitive}\\
U(\R^{\infty})\simeq \star \hskip 0.2in \mathrm{otherwise}
\end{cases}$$
which completes the proof.
\end{proof}

\subsection{$p$-local equivalences and Symmetric powers}
It is well-known that the homotopy category of spectra carries a \emph{$p$-localization} endofunctor $L_p(-)$, and for every spectra $\mathbf{X}$, there is a map $\mathbf{X} \rightarrow L_p\mathbf{X}$. For our purposes, here is the definition of the $p$-localization we need.
\begin{definition}\label{definition:plocal}
A spectrum $\mathbf{X}$ is \emph{$p$-locally contractible} if $\tilde{H}_*(\mathbf{X};\Z_{(p)})\cong 0$. A genuine $G$-spectrum $\mathbf{X}$ is $p$-locally contractible if $\Phi^H\mathbf{X}$ is $p$-locally contractible for every subgroup $H\subseteq G$. A map $f:\mathbf{X} \rightarrow \mathbf{Y}$ is a \emph{$p$-local equivalence} if the cofiber is $p$-locally contractible.

Note that for CW complexes with finitely many cells in each dimension, a $\Z_{(p)}$-homology isomorphism is the same as an $\F_p$-homology isomorphism.
\end{definition}

\begin{proposition}\label{prop:earlierlayers}
Let $G$ be a $p$-group. The $n$-th layer $\Sp^n(\Sigma^{\infty G}S^0)/\Sp^{n-1}(\Sigma^{\infty G}S^0)$ in the symmetric powers of $\Sigma^{\infty G}S^0$ is $p$-locally contractible for $n=2, 3, \ldots, p-1$. 
\end{proposition}
\begin{proof}
Pick some positive integer $\ell$, and consider the $n$-th layer $\Sp^n(S^{\ell\rho_G})/\Sp^{n-1}(S^{\ell\rho_G}) \simeq (S^{\ell\rho_G})^{\sm n}/\Sigma_n$. By induction on the group $G$, it is sufficient to prove that
$$\tilde{H}_*(((S^{\ell\rho_G})^{\sm n}/\Sigma_n)^G;\F_p)=0, \hskip 0.2in *\le 2n.$$
Since $n<p$, there are no nontrivial homomorphisms $G\rightarrow \Sigma_n$, and therefore
$$((S^{\ell\rho_G})^{\sm n}/\Sigma_n)^G = ((S^{\ell\rho_G})^{\sm n})^G/\Sigma_n = (S^{\ell})^{\sm n}/\Sigma_n$$
As in the notation of Lemma \ref{lemma:representationfixedpoints}, let $\overline{\R^n}$ denote the reduced standard representation of $\Sigma_n$. Then
$$(S^{\ell})^{\sm n}/\Sigma_n\simeq S^{\ell}\sm S^{(\overline{\R^n})^{\ell}}/\Sigma_n$$
Since $p$ is relatively prime to the order of $\Sigma_n$,
$$\tilde{H}_*(S^{\ell}\sm S^{(\overline{\R^n})^{\ell}}/\Sigma_n;\F_p)\cong  \tilde{H}_*(S^{\ell}\sm S^{(\overline{\R^n})^{\ell}};\F_p)_{\Sigma_n}$$
which is concentrated in degree $\ell n$ and higher. Since $n\ge 2$, the result follows.

\end{proof}

\subsection{Nontransitive subgroups of $\Sigma_p$}
Example \ref{example:equivariantlensspace} says that the $(G\times \Sigma_p)$-space $E_G\mathcal{F}$ is the space of ordered configurations of $p$ (not necessarily distinct) points in $\infty\rho_G$ whose sum is zero and total length is $1$. The subspace consisting of configurations of \emph{distinct} points carries a free $\Sigma_p$ action and is the $(G\times\Sigma_p)$-space $E_G\Sigma_p$.

In this section, we prove that if $G$ is a $p$-group, then the inclusion $\Sigma^{\infty G}(B_G\Sigma_p)_+\rightarrow \Sigma^{\infty G}(E_G\mathcal{F}_+/\Sigma_p)$ is a $p$-local equivalence (Proposition \ref{prop:FandSigmap}). The proof relies on Lemma \ref{lemma:fixedpointsofaquotient}, which establishes a formula for the $G$-fixed points of a quotient space.

We must first establish two simple lemmas which are used to prove Lemma \ref{lemma:fixedpointsofaquotient}.

\begin{lemma}\label{lem:normalize1}
Let $\sigma\in\Sigma_p$ be a $p$-cycle, and let $\Phi\subset\Sigma_p$ be a nontrivial group normalized by $\sigma$. Then $\Phi$ acts transitively on $\{1, \ldots, p\}$.
\end{lemma}
\begin{proof}
Without loss of generality, let $\sigma$ be the permutation sending $1 \mapsto 2 \mapsto 3 \mapsto \cdots \mapsto p \mapsto 1$. Suppose that $\Psi$ has some nontrivial permutation $\pi$ sending $p \mapsto j$. Then $\sigma^{ij}j \pi \sigma^{-ij}$ sends $ij \mapsto (i+1)j$, for $i=1, 2, \ldots, p-1$. Since $\Psi$ is normalized by $\sigma$, these permutations $\sigma^{ij} \pi \sigma^{-ij}$ all lie in $\Psi$, and so $\Psi$ is transitive.
\end{proof}
\begin{lemma}\label{lem:normalize2}
Let $G$ be a $p$-group. If $f, f':G \rightarrow \Sigma_p$ are two distinct maps, then the group $\Gamma\subset G\times \Sigma_p$ generated by $\Gamma_f$ and $\Gamma_{f'}$ intersects $\Sigma_p$ transitively.
\end{lemma}
\begin{proof}
Pick some $g \in G$ such that $f(g) \neq f'(g)$, and moreover such that $f(g)$ is not the identity permutation. In particular, $f(g)$ must be a $p$-cycle because $G$ is a $p$-group. Then $\Gamma$ contains $(g^{-1}, f(g))(g, f'(g))=f(g)^{-1}f'(g) \neq \text{id}$. Hence, $\Gamma$ intersects $\Sigma_p$ nontrivially. Let $\Psi=\Gamma\cap \Sigma_p$: then $\Phi$ is normalized by $f(g)$ because
$$(g, f(g))\Gamma (g^{-1}, f(g)^{-1})=\Gamma \implies f(g)\Psi f(g)^{-1}=\Psi$$
Now using the fact that $f(g)$ is a $p$-cycle and Lemma \ref{lem:normalize1}, it follows that $\Phi$ is transitive.
\end{proof}

\begin{definition}(\cite{Sankar}, Definition 14)\label{definition:graphsubgroup}
Let $G$ and $\Lambda$ be any finite groups, and let $H\subseteq G$ be a subgroup. For any homomorphism $f:H \rightarrow \Lambda$, its \emph{graph} is the subgroup of $H\times \Lambda$
$$\Gamma_f:=\{(h,f(h))\; :\; h\in H\}.$$

\end{definition}

\begin{lemma}\label{lemma:fixedpointsofaquotient}
Let $G$ be a $p$-group, and let $X$ be a $(G\times \Sigma_p)$-space. For each $x\in X$, let $S_x$ denote the isotropy group of $x$. Suppose that for every point $x\in X$, the intersection $S_x\cap (1\times \Sigma_p)$ is nontransitive. Then there is a decomposition
\begin{equation}\label{equation:fixedpointsofaquotient}
(X/\Sigma_p)^G = (\coprod\limits_{f:G \rightarrow \Sigma_p}X^{\Gamma_f})/\Sigma_p
\end{equation}
where $f$ varies over all homomorphisms from $G$ to $\Sigma_p$, and $\Gamma_f\subset (G\times \Sigma_p)$ denotes the graph of $f$. Here, a permutation $\sigma$ in $\Sigma_p$ takes a point of $X^{\Gamma_f}$ to a point of $X^{\Gamma_{\sigma f\sigma^{-1}}}$.
\end{lemma}
\begin{proof}
Let $x$ be a point in the $(G\times \Sigma_p)$-space $X$, and let us suppose that the image of $x$ in $X/\Sigma_p$ is $G$-fixed. This occurs if and only if the projection of $S_x$ onto $G$ is surjective.

We will show that $S_x$ contains some graph subgroup. Let us denote $S_x\cap (1\times \Sigma_p)$ by $\Psi$. Then the group $S_x$ may be thought of as the graph of a homomorphism $G \rightarrow N_G(\Psi)/\Psi$, where $N_G(\Psi)$ is the normalizer of $\Psi$ in $G$.

By assumption, $\Psi$ is nontransitive. If $\Psi=1$, then $S_x$ automatically contains a graph subgroup. If $\Psi$ is nontrivial, then by Lemma \ref{lem:normalize1} the normalizer $N_G(\Psi)$ contains no $p$-cycles. Therefore, the order of the group $N_G(\Psi)/\Psi$ is not divisible by $p$. So there are no nontrivial homomorphisms from $G$ to $N_G(\Psi)/\Psi$. It follows that $S_x = G\times \Psi$, and therefore $S_x$ contains a graph subgroup. In either case, $S_x$ contains some graph subgroup $\Gamma_f$, and thus $x\in X^{\Gamma_f}$. It follows that
$$(X/\Sigma_p)^G = (\bigcup\limits_{f:G \rightarrow \Sigma_p}X^{\Gamma_f})/\Sigma_p.$$

All that remains is to show that if $f$ and $f'$ are two distinct homomorphisms from $G$ to $\Sigma_p$, then $X^{\Gamma_f}$ and $X^{\Gamma_{f'}}$ are disjoint. By the assumption on $X$, it suffices to show that the subgroup of $G\times \Sigma_p$ which is generated by $\Gamma_f$ and $\Gamma_{f'}$ intersects $1\times \Sigma_p$ transitively. This is Lemma \ref{lem:normalize2}.
\end{proof}

\begin{proposition}\label{prop:FandSigmap}
The inclusion of $G$-spaces $(E_G\Sigma_p)/\Sigma_p \rightarrow (E_G\mathcal{F})/\Sigma_p$ is a $p$-local equivalence on all fixed point spaces.
\end{proposition}
\begin{proof}
Let $X$ denote the mapping cone of the inclusion $E_G\Sigma_p \rightarrow E_G\mathcal{F}$. By induction on the group $G$, it suffices to prove that $\tilde{H}_*((X/\Sigma_p)^G;\F_p)=0$. By Lemma \ref{lemma:fixedpointsofaquotient},
$$(X/\Sigma_p)^G \simeq (\bigvee\limits_{f:G \rightarrow \Sigma_p}X^{\Gamma_f})/\Sigma_p$$
But for an arbitrary subgroup $\Gamma\subset (G\times \Sigma_p)$,
$$X^{\Gamma}\simeq \begin{cases}
\star \hskip 0.2in \mathrm{if} \hskip 0.1in \Gamma\cap \Sigma_p=1\\
S^0 \hskip 0.1in \mathrm{if} \hskip 0.1in \Gamma\cap \Sigma_p \hskip 0.1in \mathrm{nontransitive}\, \mathrm{and} \, \mathrm{nonempty}\\
\star \hskip 0.2in \mathrm{if} \hskip 0.1in \Gamma\cap \Sigma_p \hskip 0.1in \mathrm{transitive}
\end{cases}$$
It therefore follows that $X^{\Gamma_f}\simeq \star$ for any graph subgroup $\Gamma_f$, and so $\tilde{C}_*(\bigvee\limits_{f:G \rightarrow \Sigma_p}X^{\Gamma_f};\F_p)$ is acyclic. Any point in the pointed $\Sigma_p$-space $\bigvee\limits_{f:G \rightarrow \Sigma_p}X^{\Gamma_f}$ has isotropy group nontransitive, and therefore $\tilde{C}_*(\bigvee\limits_{f:G \rightarrow \Sigma_p}X^{\Gamma_f};\F_p)$ is a projective $\F_p[\Sigma_p]$-module. It then follows that $\tilde{C}_*((\bigvee\limits_{f:G \rightarrow \Sigma_p}X^{\Gamma_f})/\Sigma_p;\F_p)$ is acyclic, as desired.

\end{proof}

\subsection{Equivalence of Classifying Spaces}
We prove a proposition analogous to a well-known nonequivariant statement, namely that the map $B_G\Aff_1 \rightarrow B_G\Sigma_p$ is a $p$-local equivalence on all fixed point spaces.

Let $G$ and $\Lambda$ be groups, and let $X$ be a $(G\times \Lambda)$-space. For any subgroup $\Psi\subseteq \Lambda$, let $C_{\Lambda}(\Psi)$ denote its centralizer. In (\cite{Sankar}, Definition 14), the following formula is given
$$(X\times_{\Lambda}E_G\Lambda)^G \simeq \coprod\limits_{[f]\in \mathrm{Hom}(G,\Lambda)/\Lambda}(X^{\Gamma_f})_{hC_{\Lambda}(\im f)}.$$
Note that this formula is a special case of Equation \ref{equation:fixedpointsofaquotient}. Specializing to the case $X=\star$, we deduce that
\begin{equation}\label{equation:fixedpointseqclassifyingspace}
(B_G\Lambda)^G\simeq \coprod\limits_{[f]\in \mathrm{Hom}(G,\Lambda)/\Lambda}BC_{\Lambda}(\im f).
\end{equation}

\begin{proposition}\label{prop:affsigma}
Let $G$ be a $p$-group. The $G$-space map $B_G\Aff_1 \rightarrow B_G\Sigma_p$ induced by the inclusion $\iota:\Aff_1 \hookrightarrow \Sigma_p$ is a $p$-local equivalence.
\end{proposition}
\begin{proof}
By induction on the group $G$, it suffices to check that the map of $G$-fixed points is an $\F_p$-homology isomorphism. The base case, $G=\{1\}$, is equivalent to proving that $B\iota:B\Aff_1\rightarrow B\Sigma_p$ is a mod $p$ homology equivalence. This is an immediate consequence of (\cite{AM}, Theorem 5.5).

Let $G$ be any $p$-group. By Equation \ref{equation:fixedpointseqclassifyingspace}, the map $(B_G\iota)^G:(B_G\Aff_1)^G \rightarrow (B_G\Sigma_p)^G$ is given by
$$(B_G\iota)^G:\coprod\limits_{[f]\in \mathrm{Hom}(G,\Aff_1)/\Aff_1}BC_{\Aff_1}(\im f) \rightarrow \coprod\limits_{[f]\in \mathrm{Hom}(G,\Sigma_p)/\Sigma_p}BC_{\Sigma_p}(\im f).$$
Every nontrivial homomorphism $G\rightarrow \Aff_1$ or $G \rightarrow \Sigma_p$ factors through a quotient of $G$ isomorphic to $\Z/p$, so we may assume $G=\Z/p$. We now describe conjugacy classes of homomorphisms from $\Z/p$ to each of $\Aff_1$ and $\Sigma_p$.
\begin{itemize}
\item The nontrivial homomorphisms $f: \Z/p \rightarrow \Aff_1$ are all conjugate, and each has centralizer equal to $\im(f)$. The trivial homomorphism has centralizer $\Aff_1$.
\item The nontrivial homomorphisms $f: \Z/p \rightarrow \Sigma_p$ are all conjugate, and each has centralizer equal to $\im(f)$. The trivial homomorphism has centralizer $\Sigma_p$.
\end{itemize}
Thus, the map $(B_{\Z/p}\iota)^{\Z/p}:(B_{\Z/p}\Aff_1)^{\Z/p} \rightarrow (B_{\Z/p}\Sigma_p)^{\Z/p}$ is given by
$$(B_{\Z/p}\iota)^{\Z/p}:B\Z/p \sqcup B\Aff_1 \rightarrow B\Z/p \sqcup B\Sigma_p,$$
The map on the first summand is the identity. The map on the second summand is $B\iota:B\Aff_1 \rightarrow B\Sigma_p$, which is an $\F_p$-homology isomorphism  (\cite{AM}, Theorem 5.5).

\end{proof}
\vskip 0.1in
{\bf Note:} Just as in the nonequivariant case, $\Sigma^{\infty G}(B_G\Sigma_p)_+$ is a stable summand of $\Sigma^{\infty G}(B_G\Z/p)_+$ with the inclusion map given by the transfer.

\section{Mod $p$ symmetric powers and Steinberg summands}
\label{sec:modpsymmetricpowersandSteinbergsummands}
Let $G$ be a $p$-group.
\begin{definition}\label{definition:M_G(n)}
For every $n\ge 0$, define the genuine $G$-spectrum $M_G(n)$ by
$$M_G(n) := S^{-n}\sm \Sp_{\Z/p}^{p^n}(\Sigma^{\infty G}S^0)/\Sp_{\Z/p}^{p^{n-1}}(\Sigma^{\infty G}S^0).$$
When $G$ is the trivial group, we simply write $M(n)$. Note that for any $i, j\ge 0$, the product maps on mod $p$ symmetric powers (Definition \ref{definition:modpsymmetricpowers}) give rise to product maps
$$M_G(i) \sm M_G(j) \rightarrow M_G(i+j).$$
\end{definition}

\begin{definition}\label{definition:genuineSteinbergsummand}
Recall from (\cite{Sankar}, Definition 12) that for any pointed $(G\times \GL_n)$-space $X$, its \emph{Steinberg summand} $e_nX$ is the na{\"i}ve $G$-spectrum
$$e_nX:=(\Sigma^{1-n}\Br_n^{\Diamond}\sm X)\sm_{\GL_n} (E_G\GL_n)_+.$$
Define the genuine $G$-spectrum $\mathbf{e}_nX$ by promoting $e_nX$ via Definition \ref{definition:promotingGspectra}, i.e.
$$\mathbf{e}_nX:=i_*(e_nX).$$
\end{definition}

We construct an equivalence between the genuine $G$-spectrum $M_G(n)$ and the Steinberg summand of the equivariant classifying space $B_G(\Z/p)_+^n$, i.e.
\setcounter{theorem}{0}
\begin{theorem}
For every integer $n\ge 1$, there is an equivalence of genuine $G$-spectra
$$M_G(n) \simeq \mathbf{e}_nB_G(\Z/p)_+^n.$$
\end{theorem}
\setcounter{theorem}{57}
We provide the proof here, referencing the supporting computational results proven in this section.
\begin{proof}
Proposition \ref{prop:firstcofiber} states that there is an equivalence of genuine $G$-spectra
$$M_G(1)\simeq \Sigma_+^{\infty G}B_G\Aff_1.$$
Because $\Br_1^{\Diamond}\simeq S^0$, it follows immediately from the definitions that there is an equivalence of genuine $G$-spectra $\mathbf{e}_1B_G(\Z/p)_+ \simeq \Sigma^{\infty G}(B_G\Aff_1)_+.$ Therefore, Proposition \ref{prop:firstcofiber} is the $n=1$ case of Theorem \ref{thm:maintheorem}, namely
\begin{equation}\label{eq:firstcofiber}
M_G(1) \simeq \mathbf{e}_1B_G\Z/p_+.
\end{equation}
The proof of Theorem \ref{thm:maintheorem} relies on the follow diagram
\begin{equation}\label{eq:maindiagram}
\xymatrix{\mathbf{e}_nB_G(\Z/p)_+^n\ar[r] & (\mathbf{e}_1B_G\Z/p_+)^{\sm n}\ar@{=}[r] & M_G(1)^{\sm n}\ar[r] & M_G(n)}
\end{equation}
where the first map is the inclusion of the Steinberg summand, the middle equivalence is from Equation \ref{eq:firstcofiber}, and the last map is the product map. In diagram \ref{eq:maindiagram} above, we wish to show that the composition is a $p$-local equivalence on all geometric fixed point spectra. By induction on the group $G$, it will suffice to check that it is a $p$-local equivalence on $G$-geometric fixed points, and this can be checked at the level of $\F_p$-homology. That is, we wish to show that the induced map
$$f: H_*(\Phi^G(\mathbf{e}_nB_G(\Z/p)_+^n); \F_p) \rightarrow H_*(\Phi^G(M_G(n)); \F_p)$$
is an isomorphism of graded $\F_p$-vector spaces. This is Corollary \ref{corollary:compositionisanisomorphismonhomology2}.
\end{proof}

The proof of Corollary \ref{corollary:compositionisanisomorphismonhomology2} may be outlined as follows. The geometric fixed point spectra $\Phi^G(\mathbf{e}_nB_G(\Z/p)_+^n)$ and $\Phi^G(M_G(n))$ have been given explicit decompositions (discussed in Section 6.1, Proposition \ref{prop:geometricfixedpointsequivalence}) indexed over the subgroups $H\in\mathcal{C}$. Thus in Section 6.3 we present bases for the two $\F_p$-vector spaces $H_*(\Phi^G(\mathbf{e}_nB_G(\Z/p)_+^n); \F_p)$ and $H_*(\Phi^G(M_G(n)); \F_p)$ and the map $f$ all in terms of matrices. Because the two groups $H_*(\Phi^G(\mathbf{e}_nB_G(\Z/p)_+^n); \F_p)$ and $H_*(\Phi^G(M_G(n)); \F_p)$ are abstractly isomorphic, it suffices to prove that $f$ has trivial kernel, which boils down to a linear algebra problem which is done in Section 6.2, Corollary \ref{corollary:compositionisanisomorphismonhomology}.

\subsection{$H$-summands}
\label{subsec:Hsummands}
Let $G$ be a $p$-group, and let $\mathcal{C}$ denote the poset of subgroups $H\unlhd G$ such that $G/H$ is an elementary abelian $p$-group (Definition \ref{definition:theposetC}).
\begin{definition}
For every $H\in\mathcal{C}$, let $d(H)$ denote the rank of $G/H$ as an $\F_p$-vector space. Note that two subgroups $H, K\in\mathcal{C}$ are transverse (Definition \ref{definition:transverse}) if and only if $d(H)+d(K)=d(H\cap K)$.
\end{definition}
Proposition \ref{prop:modpstablefinitesymmetricpowers} implies that, for every $N\ge 1$ there is an equivalence of spectra
$$\Phi^G\Sp_{\Z/p}^N(\Sigma^{\infty G}S^0)\simeq \bigvee\limits_{H\in\mathcal{C}}\Sp_{\Z/p}^{\lfloor N/p^{d(H)}\rfloor}(\Sigma^{\infty}S^0)\sm \Pri^{G/H}(S^{\infty\rreg_G})$$
which is suitably compatible with the inclusions $\Sp^{N-1}(-) \rightarrow \Sp^N(-)$. Taking $N=p^n$, we immediately deduce the formula
\begin{equation}\label{eqn:geomfixedptsM_G(n)}
\begin{aligned}
\Phi^GM_G(n)&\simeq \bigvee\limits_{H\in\mathcal{C}}M_G(n-d(H))\sm \Sigma^{-d(H)}\Pri^{G/H}(S^{\infty\rreg_G})\\
&\simeq \bigvee\limits_{H\in\mathcal{C}}M_G(n-d(H))\sm \Sigma^{1-d(H)}\Po(G)_{\supset H}^{\Diamond}\sm B(G/H)_+
\end{aligned}
\end{equation}
\begin{definition}\label{definition:H-summand}
Let $n$ be a positive integer and $H\in\mathcal{C}$ be a subgroup of $G$. The spectrum
$$M_G(n,H):=M(n-d(H))\sm \Sigma^{-d(H)}\Pri^{G/H}(S^{\infty\rreg_G})$$
is called the \emph{$H$-summand of $\Phi^GM_G(n)$.} For any two positive integers $m,n$ and subgroups $H,K\in\mathcal{C}$, there is a product map
$$M_G(m,H)\sm M_G(n,K) \rightarrow M_G(m+n,H\cap K)$$
which is determined by the product maps
$$\Phi^GM_G(m)\sm \Phi^GM_G(n) \rightarrow \Phi^GM_G(m+n).$$
\end{definition}

\begin{proposition}\label{prop:geometricfixedpointsequivalence}
Let $G$ be a $p$-group and let $n$ be any positive integer. For every subgroup $H\in\mathcal{C}$, there is an equivalence of $H$-summand spectra
\begin{equation}\label{eqn:geometricfixedptsequivalence}
M_G(n,H) \simeq E_n(H).
\end{equation}
where $E_n(H)$ is the $H$-summand of the fixed point spectrum $(e_nB_G(\Z/p)_+^n)^G$ (\cite{Sankar}, Definition 20).

Let $m, n$ be positive integers and suppose that $H,K\in\mathcal{C}$ are transverse. Then there is a commutative diagram
$$\xymatrix{M_{G,m}(H)\sm M_{G,n}(K) \ar[r]\ar@{=}[d] & \ar@{=}[d]M_{G,m+n}(H\cap K)\\
E_m(H)\sm E_n(K)\ar[r] & E_{m+n}(H\cap K)}$$
where the rows are the products on summands, and the columns are the equivalences just described.
\end{proposition}
\begin{proof}
By Proposition \ref{prop:stableprimitivespgroup},
$$\Pri^{G/H}(S^{\infty\rreg_G})\simeq \Sigma \Po(G)_{\supset H}^{\Diamond}\sm B(G/H)_+.$$
When $G/H$ is elementary abelian, the subgroup complex $\Po(G)_{\supset H}$ is identical to the flag complex of the $\F_p$-vector space $G/H\cong (\Z/p)^{d(H)}$. By (\cite{MP}, Theorem A), the layer $M(n-d(H))$ is $p$-locally equivalent to $e_{n-d(H)}B(\Z/p)_+^{n-d(H)}$. Thus,
$$\begin{aligned}
M_G(n,H) &:= M(n-d(H))\sm \Sigma^{-d(H)}\Pri^{G/H}(S^{\infty\rreg_G})\\
&\simeq e_{n-d(H)}B(\Z/p)_+^{n-d(H)}\sm \Sigma^{1-d(H)}\Br_{d(H)}^{\Diamond}\sm B(G/H)_+\\
&=: E_n(H).
\end{aligned}$$
When $H$ and $K$ are transverse, the subgroup complex product (Definition \ref{definition:subgroupcomplexproduct}) is identical to the flag complex product (\cite{Sankar}, above Proposition 10). Therefore, by combining Proposition \ref{prop:transversecase} with (\cite{Sankar}, Proposition 21), we deduce that the equivalence $\Sigma^{-d(H)}\Pri^{G/H}(S^{\infty\rreg_G})\simeq \Sigma^{1-d(H)}\Br_{n-d(H)}^{\Diamond}\sm B(G/H)_+$ respects the products on both sides.

The equivalence $M(n-d(H))\simeq e_{n-d(H)}B(\Z/p)_+^{n-d(H)}$ constructed in (\cite{MP}, Theorem A) was built so as to respect the product structures on both sides. Therefore, the equivalence $M_G(n,H)\simeq E_n(H)$ respects the product structures on both sides.
\end{proof}
We note the following corollary.
\begin{corollary}\label{corollary:geomfixedptsequivalence}
There is an equivalence of spectra
$$\Phi^GM_G(n) \simeq (e_nB_G(\Z/p)_+^n)^G.$$
\end{corollary}
\begin{proof}
Combine Equation \ref{eqn:geometricfixedptsequivalence}, Equation \ref{eqn:geomfixedptsM_G(n)}, and (\cite{Sankar}, Definition 20).
\end{proof}
\subsection{Matrices with transverse row nullspaces}
\label{subsec:Matriceswithtransverserownullspaces}

\begin{definition}\label{definition:Mat_{n,r}}
Let $n, r\ge 0$ be nonnegative integers. Then we write $\Mat_{n,r}=\mathrm{Hom}((\Z/p)^r,(\Z/p)^n)$ for the set of $n\times r$ matrices with entries in the field $\F_p$. For each subspace $V\subseteq (\Z/p)^r$, let $\Mat_{n,r}(V)\subset \Mat_{n,r}$ denote the set of $n\times r$ matrices with nullspace $V$.
\end{definition}
\begin{definition}\label{definition:matrixtransverse}
Let $\mathcal{T}\subset \Mat_{n,r}$ denote the set of $n\times r$ matrices with the following property: if a matrix $A\in\mathcal{T}$ has exactly $k$ nonzero rows for some $k$, then those nonzero row vectors are linearly independent. Let $\mathcal{T}(V)\subset \Mat_{n,r}(V)$ denote the intersection $\mathcal{T}(V)=\mathcal{T}\cap \Mat_{n,r}(V)$. The set $\mathcal{T}(V)$ is equivalently characterized as the set of $n\times r$ matrices with nullspace $V$ and exactly $s$ nonzero rows, where $\dim(V)=r-s$.
\end{definition}


The $\GL_n$-set $\Mat_{n,r}$ decomposes as
$$\Mat_{n,r}=\bigsqcup\limits_{V\subset (\Z/p)^r}\Mat_{n,r}(V).$$
Fix a subspace $V\subseteq (\Z/p)^r$, and write $\dim(V)=r-s$. Our goal in this section is to prove the Proposition \ref{prop:Steinbergcomposition} below.
\begin{proposition}\label{prop:Steinbergcomposition}
Let $D$ be any finite-dimensional $\F_p[\GL_n]$-module. Let $A$ denote the $\F_p[\GL_n]$-module $A=\bigoplus\limits_{\Mat_{n,r}(V)}D$. The composition
$$\xymatrix{e_nA\ar@{^{(}->}[r] & A\ar[rr]^{\proj_{\T(V)}} && A\ar[r]^{e_n(-)} & e_nA}$$
is a monomorphism of $\F_p$-vector spaces. Therefore, because the source and target have the same dimension, the composition above is an isomorphism.
\end{proposition}
Before we discuss the proof of Proposition \ref{prop:Steinbergcomposition}, we record a corollary, which is used in the proof of Corollary \ref{corollary:compositionisanisomorphismonhomology2} and thus of Theorem \ref{thm:maintheorem}.
\begin{corollary}\label{corollary:compositionisanisomorphismonhomology}
Let $D$ denote the $\F_p[\GL_n]$ module $D=H_*(B(\Z/p)^n;\F_p)$. The following composition is an isomorphism
$$\xymatrix{e_n\bigoplus\limits_{\Mat_{n,r}}D \ar@{^{(}->}[r] & e_1^{\boxtimes n}\bigoplus\limits_{\Mat_{n,r}}D\ar[r]^{\proj_{\T}} & e_1^{\boxtimes n}\bigoplus\limits_{\Mat_{n,r}}D\ar[r] & e_n\bigoplus\limits_{\Mat_{n,r}}D}.$$
\end{corollary}
\begin{proof}
The decomposition $\Mat_{n,r}=\bigsqcup\limits_{V\subset (\Z/p)^r}\Mat_{n,r}(V)$ is $\GL_n$-equivariant, so it suffices to prove the same statement with $\Mat_{n,r}$ replaced by $\Mat_{n,r}(V)$. Next, we observe that the following diagram commutes:
$$\xymatrix{e_n\bigoplus\limits_{\Mat_{n,r}}D \ar@{^{(}->}[r]\ar@{^{(}->}[dr] & \bigoplus\limits_{\Mat_{n,r}}D\ar[r]^{\proj_{\T}} & \bigoplus\limits_{\Mat_{n,r}}D\ar[r] & e_n\bigoplus\limits_{\Mat_{n,r}}D\\
& e_1^{\boxtimes n}\bigoplus\limits_{\Mat_{n,r}}D\ar[r]^{\proj_{\T}}\ar@{^{(}->}[u] & e_1^{\boxtimes n}\bigoplus\limits_{\Mat_{n,r}}D\ar[ur]\ar@{^{(}->}[u] & }.$$
Now the proof statement is an immediate corollary of (\cite{Sankar}, Proposition \ref{prop:Steinbergcomposition}).
\end{proof}

We now prove Proposition \ref{prop:Steinbergcomposition}. We need some notation and lemmas.
\begin{definition}
Let $\Upsilon\subset \mathcal{T}(V)$ denote the subset of matrices whose last $s$ rows are nonzero. Let $N$ be the free $\F_p[\GL_n]$-module
$$N=\bigoplus\limits_{\Mat_{n,r}(V)}\F_p[\GL_n].$$
Associated to the subsets $\Upsilon \subset \T(V) \subset \Mat_{n,r}(V)$ are endomorphisms $\proj_{\T(V)}$ and $ \proj_{\Upsilon}$ of the vector space $N$ such that
$$\proj_{\T(V)}\circ\proj_{\Upsilon}=\proj_{\Upsilon}=\proj_{\Upsilon}\circ\proj_{\T(V)}.$$
\end{definition}

Recall from (\cite{Sankar}, Section 1) that $B_n$ (resp. $\Sigma_n$) denotes the group of invertible $n\times n$ upper triangular matrices (resp. permutation matrices). The elements $\overline{B}_n$ and $\overline{\Sigma}_n$ of the group algebra $\Z_{(p)}[\GL_n]$ define endomorphisms of the $\F_p[\GL_n]$-module $N$. The conjugate Steinberg idempotent $\hat{e}_n$ is the element of the group ring $\Z_{(p)}[\GL_n]$ defined by $\hat{e}_n=\frac{1}{c_n}\cdot\overline{B}_n\overline{\Sigma}_n$, where $c_n\in\Z_{(p)}^{\times}$.

Let $B_{n-s}\times B_s\subset B_n$ (resp. $\Sigma_{n-s}\times \Sigma_s$) denote the set of upper-triangular matrices (resp. permutation matrices) whose $ij$-th entry is zero whenever $1\le i\le n-s$ and $n-s+1\le j\le n$. Let $\hat{e}_{n-s}\boxtimes \hat{e}_s\in \F_p[\GL_n]$ denote the image of the idempotent element $\hat{e}_{n-s}\otimes \hat{e}_s\in \F_p[\GL_{n-s}]\otimes_{\F_p} \F_p[\GL_s]$ under the block inclusion.
\begin{lemma}
The following two endomorphisms of $N$ as an $\F_p$-vector space are equal:
$$\proj_{\Upsilon}\circ\overline{B}_n\circ\proj_{\T(V)} = \overline{B_{n-s}\times B_s}\circ\proj_{\Upsilon}.$$
\end{lemma}
\begin{proof}
Let $A$ be an $n\times r$ matrix in $\T(V)$ and let $b$ be an $n\times n$ invertible upper triangular matrix. Then
$$bA\in \Upsilon \iff (b\in B_{n-s}\times B_s \; \text{and} \; A\in \Upsilon).$$
It immediately follows that
$$\proj_{\Upsilon}\circ\overline{B}_n\circ\proj_{\T(V)} = \proj_{\Upsilon}\circ\overline{B_{n-s}\times B_s}\circ\proj_{\T(V)}.$$
The set $\Upsilon$ is preserved by the group $\GL_{n-s}\times \GL_s$, and so it follows that
$$\begin{aligned}
\proj_{\Upsilon}\circ\overline{B_{n-s}\times B_s}\circ\proj_{\T(V)}&= \overline{B_{n-s}\times B_s}\circ\proj_{\Upsilon}\circ \proj_{\T(V)}\\
&= \overline{B_{n-s}\times B_s}\circ\proj_{\Upsilon}.
\end{aligned}$$
\end{proof}
\begin{lemma}
The composition $\xymatrix{\hat{e}_nN\ar@{^{(}->}[r] & N\ar[r]^{\proj_{\Upsilon}} & N}$ is a monomorphism of $\F_p$-vector spaces.
\end{lemma}
\begin{proof}
For any sub-$\F_p$-vector space $L\subseteq N$, we write $\proj_{\Upsilon}(L)$ to denote the image of $L$ under the endomorphism $\proj_{\Upsilon}$. There is an inclusion of $\F_p$-vector spaces
$$\proj_{\Upsilon}(\hat{e}_nN)\subseteq \proj_{\Upsilon}((\hat{e}_{n-s}\boxtimes \hat{e}_s)N).$$
We will prove that these two vector spaces are equal. Because the set $\Gamma\subset \Mat_{n,r}(V)$ is preserved by the subgroup $\GL_{n-s}\times \GL_s\subset \GL_n$, it follows that
$$\proj_{\Upsilon}((\hat{e}_{n-s}\boxtimes \hat{e}_s)N)=\proj_{\Upsilon}\proj_{\Upsilon}((\hat{e}_{n-s}\boxtimes \hat{e}_s)N)=\proj_{\Upsilon}((\hat{e}_{n-s}\boxtimes \hat{e}_s)\proj_{\Upsilon}N).$$
Therefore the $\F_p$-vector space $\proj_{\Upsilon}((\hat{e}_{n-s}\boxtimes \hat{e}_s)N)$ is spanned by elements of the form
$$\proj_{\Upsilon}((\hat{e}_{n-s}\boxtimes \hat{e}_s)(A\otimes x)), \hskip 0.3in A\in \Upsilon, x\in \GL_n.$$
Pick such an element. Let $b$ be any element of the group of block upper triangular matrices $B_{n-s}\times B_s$. Then for every $\sigma \in \Sigma_n$, we have
$$\sigma\cdot b\cdot A\in \Upsilon \iff \sigma\in \Sigma_{n-s}\times \Sigma_s.$$
Therefore, it follows that
$$\proj_{\Upsilon}(\overline{\Sigma_{n-s}\times \Sigma_s}\cdot\overline{B_{n-s}\times B_s}\cdot (A\otimes x))=\proj_{\Upsilon}(\overline{\Sigma}_n\cdot\overline{B_{n-s}\times B_s}\cdot (A\otimes x)).$$
Now let $b$ be any element of the group $B_n$ such that $b\notin B_{n-s}\times B_s$. Then it follows that $b\cdot A\notin \T(V)$. So for any permutation matrix $\sigma\in\Sigma_n$, we have $\sigma\cdot b\cdot A\notin \T(V)$. In particular, $\sigma\cdot b\cdot A\notin \Upsilon$. Therefore
$$\proj_{\Upsilon}(\overline{\Sigma}_n\cdot\overline{B_{n-s}\times B_s}\cdot (A\otimes x))=\proj_{\Upsilon}(\overline{\Sigma}_n\cdot\overline{B}_n\cdot (A\otimes x))=\proj_{\Upsilon}(\hat{e}_n(A\otimes x)).$$
We therefore conclude that
$$\proj_{\Upsilon}((\hat{e}_{n-s}\boxtimes \hat{e}_s)N)\subseteq \proj_{\Upsilon}(\hat{e}_nN),$$
and therefore the two $\F_p$-vector spaces above are equal.

We now prove that $\dim_{\F_p}(\hat{e}_nN)=\dim_{\F_p}(\proj_{\Upsilon}((\hat{e}_{n-s}\boxtimes \hat{e}_s)N))$. From this fact, we will conclude that $\hat{e}_nN$ and $\proj_{\Upsilon}(\hat{e}_nN)$ have the same dimension as $\F_p$-vector spaces, which completes the proof. Because $N=\bigoplus\limits_{\Mat_{n,r}(V)}\F_p[\GL_n]$ is a free $\F_p[\GL_n]$-module, we calculate
$$\begin{aligned}
\dim_{\F_p}(\hat{e}_nN)=\dim_{\F_p}(e_nN)&= \dim_{\F_p}(\St_n\otimes_{\GL_n}N)\\
&=\dim_{\F_p}(\St_n)\cdot |\Mat_{n,r}(V)|\\
&= \boxed{p^{\binom{n}{2}}\prod\limits_{i=0}^{s-1}(p^n-p^i)
}.\end{aligned}$$
Because $\proj_{\Upsilon}N=\bigoplus\limits_{\Upsilon}\F_p[\GL_n]$ is a free $\F_p[\GL_{n-s}\times \GL_s]$-module, we calculate
$$\begin{aligned}
\dim_{\F_p}(\proj_{\Upsilon}((\hat{e}_{n-s}\boxtimes \hat{e}_s)N))&= \dim_{\F_p}((\hat{e}_{n-s}\boxtimes \hat{e}_s)\proj_{\Upsilon}N)\\
&= \dim_{\F_p}(\St_{n-s}\otimes \St_s)\cdot |\Upsilon|\cdot \frac{|\GL_n|}{|\GL_{n-s}\times \GL_s|}\\
&= \dim_{\F_p}(\St_{n-s}\otimes \St_s)\cdot |\GL_s|\cdot \frac{|\GL_n|}{|\GL_{n-s}\times \GL_s|}\\
&= \boxed{p^{\binom{n-s}{2}+\binom{s}{2}}\cdot \frac{\prod\limits_{i=0}^{n-1}(p^n-p^i)}{\prod\limits_{i=0}^{n-s-1}(p^{n-s}-p^i)}}.
\end{aligned}$$
It is routine to check that the two boxed expressions are equal.
\end{proof}

\begin{proof}[Proof of Proposition \ref{prop:Steinbergcomposition}]
Any such $D$ receives a surjective map from a finite-dimensional free $\F_p[\GL_n]$-module. Therefore it suffices to consider the case $D=\F_p[\GL_n]$. In this case, $A$ is equal to the module $N=\bigoplus\limits_{\Mat_{n,r}(V)}\F_p[\GL_n]$ defined earlier.

Let $x\in N$ be any element such that $e_nx$ is nonzero. We wish to prove that $e_n(\proj_{\T(V)}(e_nx))$ is nonzero. We will prove that $\proj_{\Upsilon}(e_n(\proj_{\T(V)}(e_nx)))$ is nonzero.
$$\begin{aligned}
\proj_{\Upsilon}(e_n(\proj_{\T(V)}(e_nx)))&:= \proj_{\Upsilon}\overline{B}_n\overline{\Sigma}_n\proj_{\T(V)}\overline{B}_n\overline{\Sigma}_nx\\
&= \proj_{\Upsilon}\overline{B}_n\proj_{\T(V)}\overline{\Sigma}_n\overline{B}_n\overline{\Sigma}_nx \hskip 0.3in (\T(V) \text{ is } \Sigma_n \text{--invariant})\\
&= \overline{B_{n-s}\times B_s}\proj_{\Upsilon}\overline{\Sigma}_n\overline{B}_n\overline{\Sigma}_nx. \hskip 0.3in (\text{Lemma 3.3})
\end{aligned}$$
Recall from \ref{sec:defineSteinberg} that for any $\Z_{(p)}[\GL_n]$-module $M$, the linear operators
$$\overline{B}_n:\xymatrix{\hat{e}_nM\ar@/^1ex/[r]&  e_nM\ar@/^1ex/[l]}:\overline{\Sigma}_n$$
are inverse isomorphisms. Therefore, because the element $e_nx=\overline{B}_n\overline{\Sigma}_nx$ was assumed to be nonzero, it follows that the element $\overline{\Sigma}_n\overline{B}_n\overline{\Sigma}_nx$ is nonzero. Therefore, by Lemma 3.2, the element $\proj_{\Upsilon}\overline{\Sigma}_n\overline{B}_n\overline{\Sigma}_nx$ is nonzero.

The conjugate Steinberg idempotent $\hat{e}_n$ can be written in the form
$$\hat{e}_n=(\hat{e}_{n-s}\boxtimes \hat{e}_s)\overline{U}_{n-s,s}\overline{\Sigma_{\mathrm{shuf}}(n-s,s)}.$$
It therefore follows that the summand $\hat{e}_nN=\overline{\Sigma}_n\overline{B}_nN$ is a sub-$\F_p$-vector space of $(\hat{e}_{n-s}\boxtimes \hat{e}_s)N$. Therefore, the element $\overline{\Sigma}_n\overline{B}_n\overline{\Sigma}_nx$ is an element of the vector space $(\hat{e}_{n-s}\boxtimes \hat{e}_s)N$. The set $\Upsilon$ is preserved by the group $\GL_{n-s}\times \GL_s$, and so it follows that the endomorphisms $\proj_{\Upsilon}$ and $\hat{e}_{n-s}\boxtimes \hat{e}_s$ commute. Therefore, $\proj_{\Upsilon}\overline{\Sigma}_n\overline{B}_n\overline{\Sigma}_nx$ lies in the vector space $(\hat{e}_{n-s}\boxtimes \hat{e}_s)N$.

The linear operator
$$\overline{B_{n-s}\times B_s}:(\hat{e}_{n-s}\boxtimes \hat{e}_s)N \rightarrow (e_{n-s}\boxtimes e_s)N$$
is an isomorphism. Therefore, because $\proj_{\Upsilon}\overline{\Sigma}_n\overline{B}_n\overline{\Sigma}_nx$ is a nonzero element of the vector space $(\hat{e}_{n-s}\boxtimes \hat{e}_s)N$, we conclude that $\overline{B_{n-s}\times B_s}\proj_{\Upsilon}\overline{\Sigma}_n\overline{B}_n\overline{\Sigma}_nx$ is nonzero, as desired.
\end{proof}

\subsection{Proof of Theorem 1}
\label{subsec:proofofthemaintheorem}
Let $G$ be a $p$-group, and let $n, i$ be positive integers such that $n\ge i$. For any pointed $(G\times\GL_n)$-space $X$, there are inclusions and projections of Steinberg summands as defined in (\cite{Sankar}, Sections 2.2 and 2.3),
$$e_nX \rightarrow (e_i\boxtimes e_{n-i})X \hskip 0.2in \text{and} \hskip 0.2in (e_i\boxtimes e_{n-i})X \rightarrow e_nX.$$
We thus obtain an inclusion of na{\"i}ve $G$-spectra $e_nX \rightarrow e_1^{\boxtimes n}X$.

We now specialize to the case where $X$ is the $G$-equivariant classifying space with a disjoint basepoint, $X=B_G(\Z/p)_+^n$. In this situation, the fixed point spectrum $(e_nB_G(\Z/p)_+^n)^G$ was given a complete description in (\cite{Sankar}, Section 4.2). We express that description in terms of matrices for the purposes of our computation.
\begin{definition}
Let $G$ be a finite $p$-group. We denote by $F$ the minimal subgroup of $G$ such that $G/F$ is elementary abelian. We let $r$ denote the rank of $G/F$, and we fix an isomorphism $G/F\cong (\Z/p)^r$. The subgroups $H\in\mathcal{C}$ are in bijective correspondence with linear subspaces $V\subseteq (\Z/p)^r$, and we let $\cd(V)$ denote the codimension of $V$.
\end{definition}
Thus,
\begin{equation}\label{eqn:Steinbergfixedpointsmatrix}
\begin{aligned}
(e_nB_G(\Z/p)_+^n)^G &\simeq e_n\bigvee\limits_{\Hom(G,(\Z/p)^n)}B(\Z/p)_+^n\\
&\simeq e_n\bigvee\limits_{\Mat_{n,r}}B(\Z/p)_+^n
\end{aligned}
\end{equation}
There is an equivalence of $(\GL_1\times\cdots\times\GL_1)$-sets given by using the $i$-th input row vector as the $i$-th row of the output matrix for $i=1, 2, \ldots, n$,
$$\Mat_{1,r}\times\cdots\times\Mat_{1,r} \cong \Mat_{n,r}.$$
Then the Steinberg inclusion $(e_nB_G(\Z/p)_+^n)^G \rightarrow (e_1^{\boxtimes n}B_G(\Z/p)_+^n)^G$ is described by the inclusion of Steinberg summands
$$e_n\bigvee\limits_{\Mat_{n,r}}B(\Z/p)_+^n \rightarrow e_1^{\boxtimes}\bigvee\limits_{\Mat_{n,r}}B(\Z/p)_+^n\simeq (e_1\bigvee\limits_{\Mat_{1,r}}B(\Z/p)_+)^{\sm n}.$$
Expressed on the level of homology, the map $\tilde{H}_*(\Phi^G\mathbf{e}_nB_G(\Z/p)_+^n;\F_p) \rightarrow \tilde{H}_*(\Phi^G\mathbf{e}_1^{\boxtimes}B_G(\Z/p)_+^n;\F_p)$ is given by the inclusion
\begin{equation}\label{eqn:inclusionofSteinbergsummandsonhomology}
e_n\bigoplus\limits_{\Mat_{n,r}}H_*(B(\Z/p)^n;\F_p) \rightarrow e_1^{\boxtimes}\bigoplus\limits_{\Mat_{n,r}}H_*(B(\Z/p)^n;\F_p).
\end{equation}

The product $\tilde{H}_*(\Phi^GM_G(1)^{\sm n}) \rightarrow \tilde{H}_*(\Phi^GM_G(n))$ can also be described in this language. By Proposition \ref{prop:firstcofiber},
\begin{equation}\label{equation:homologyofM_G(1)^n}
\begin{aligned}
\tilde{H}_*(\Phi^GM_G(1)^{\sm n}) &\cong \tilde{H}_*(\Phi^G\mathbf{e}_1^{\boxtimes n}B_G(\Z/p)_+^n;\F_p)\\
&\cong e_1^{\boxtimes}\bigoplus\limits_{\Mat_{n,r}}H_*(B(\Z/p)^n;\F_p)\\
\end{aligned}
\end{equation}
And by Corollary \ref{corollary:geomfixedptsequivalence},

\begin{equation}\label{equation:homologyofM_G(n)}
\begin{aligned}
\tilde{H}_*(\Phi^GM_G(n)) &\cong \tilde{H}_*(\Phi^G\mathbf{e}_nB_G(\Z/p)_+^n;\F_p)\\
&\cong e_n\bigoplus\limits_{\Mat_{n,r}}H_*(B(\Z/p)^n;\F_p)
\end{aligned}
\end{equation}
\begin{proposition}\label{proposition:M_G(n)productonhomology}
Under the isomorphisms of Equations \ref{equation:homologyofM_G(1)^n} and \ref{equation:homologyofM_G(n)}, the product $\tilde{H}_*(\Phi^GM_G(1)^{\sm n}) \rightarrow \tilde{H}_*(\Phi^GM_G(n))$ on the layers in the mod $p$ symmetric powers is given by the composition
$$\xymatrix{e_1^{\boxtimes n}\bigoplus\limits_{\Mat_{n,r}}H_*(B(\Z/p)^n;\F_p)\ar[r]^{\proj_{\T}} & e_1^{\boxtimes n}\bigoplus\limits_{\Mat_{n,r}}H_*(B(\Z/p)^n;\F_p)\ar[r] & e_n\bigoplus\limits_{\Mat_{n,r}}H_*(B(\Z/p)^n;\F_p)}.$$
\end{proposition}
\begin{proof}

The $\GL_1$-set $\Mat_{1,r}$ is a disjoint union $\Mat_{1,r}=\bigsqcup\limits_{V\subseteq (\Z/p)^r}\Mat_{1,r}(V)$, where $\Mat_{1,r}(V)$ is the set of $r$-dimensional row vector with nullspace $V$. Note that $\Mat_{1,r}(V)$ is empty unless $V$ has codimension 0 or 1. Let $H_1, \ldots, H_n\in\mathcal{C}$, and let $V_1, \ldots, V_n\subseteq (\Z/p)^r$ be the associated subspaces. There is an isomorphism by using the $i$-th input row vector as the $i$-th row of the output matrix for $i=1, 2, \ldots, n$,
$$\bigsqcup\limits_{V_1, \ldots, V_n\subseteq (\Z/p)^r}\Mat_{1,r}(V_1)\times\cdots\times \Mat_{1,r}(V_n) \cong \bigsqcup\limits_{V\subseteq (\Z/p)^r}\Mat_{n,r}(V).$$
which maps $\Mat_{1,r}(V_1)\times\cdots\times \Mat_{1,r}(V_n) \rightarrow \Mat_{n,r}(V_1\cap\cdots\cap V_n)$. Observe that collection of (codimension 0 or 1) subgroups $H_1, \ldots, H_n\in\mathcal{C}$ are transverse (Definition \ref{definition:transverse}) if and only if the matrices coming from this map are in $\T$ (Definition \ref{definition:matrixtransverse}).

Let $H\in\mathcal{C}$, and let $V\subseteq (\Z/p)^r$ be the associated linear subspace under the isomorphism $G/F\cong (\Z/p)^r$. By the definition of the spectrum $E_n(H)$ (\cite{Sankar}, Definition 20)
$$\begin{aligned}
\tilde{H}_*(E_n(H);\F_p) &\cong e_n\bigoplus\limits_{\Hom(G/H,(\Z/p)^n)}H_*(B(\Z/p)^n;\F_p)\\
&\cong e_n\bigoplus\limits_{\Mat_{n,r}(V)}H_*(B(\Z/p)^n;\F_p).
\end{aligned}$$

Under the isomorphism of Equation \ref{eqn:Steinbergfixedpointsmatrix}, the product of $H$-summands in the Steinberg projection $\tilde{H}_*(\Phi^G\mathbf{e}_1^{\boxtimes n}B_G(\Z/p)_+^n) \rightarrow \tilde{H}_*(\Phi^G\mathbf{e}_nB_G(\Z/p)_+^n)$, namely
\begin{equation}\label{eqn:productonE_n(H)}
\bigotimes\limits_{i=1}^{n}\tilde{H}_*(E_1(H_i);\F_p) \rightarrow \tilde{H}_*(E_n(H_1\cap\cdots\cap H_n);\F_p)
\end{equation}
is given by
$$\bigotimes\limits_{i=1}^{n}\left(e_1\bigoplus\limits_{\Mat_{1,r}(V_i)}H_*(B\Z/p;\F_p)\right) \rightarrow e_n\bigoplus\limits_{\Mat_{n,r}(V_1\cap\cdots\cap V_n)}H_*(B(\Z/p)^n;\F_p).$$

Let $H_1, \ldots, H_n\in\mathcal{C}$ be subgroups. The product $\tilde{H}_*(\Phi^GM_G(1)^{\sm n};\F_p) \rightarrow \tilde{H}_*(\Phi^GM_G(n);\F_p)$ can be described in terms of the maps on $H$-summands (Definition \ref{definition:H-summand})
\begin{equation}\label{eqn:M_G(n)product}
\bigotimes\limits_{i=1}^{n}\tilde{H}_*(M_G(1,H_i);\F_p) \rightarrow \tilde{H}_*(M_G(n,H_1\cap\cdots\cap H_n)).
\end{equation}
Note that $M_G(1,H_i)\simeq 0$ if $d(H)\ge 2$. If the subgroups $H_1, \ldots, H_n$ are transverse then by Proposition \ref{prop:geometricfixedpointsequivalence}, the map of Equation \ref{eqn:M_G(n)product} is given by Equation \ref{eqn:productonE_n(H)}. If the subgroups $H_1, \ldots, H_n$ are nontransverse then by Proposition \ref{prop:nontransversecase}, the the map of Equation \ref{eqn:M_G(n)product} is zero.
\end{proof}

\begin{corollary}\label{corollary:compositionisanisomorphismonhomology2}
The composition
$$\mathbf{e}_nB_G(\Z/p)_+^n \rightarrow (\mathbf{e}_1B_G(\Z/p)_+)^{\sm n} \simeq M_G(1)^{\sm n} \rightarrow M_G(n)$$
is an isomorphism after applying the functor $\tilde{H}_*(\Phi^G(-);\F_p)$.
\end{corollary}
\begin{proof}
Equation \ref{eqn:inclusionofSteinbergsummandsonhomology} gives us a description of the first map, and Proposition \ref{proposition:M_G(n)productonhomology} gives us a description of the second map. Now apply Corollary \ref{corollary:compositionisanisomorphismonhomology}.
\end{proof}


\section{Splitting of the filtration}
\label{sec:splittingofthefiltration}
Recall from Section \ref{subsec:equivarianteilenbergmaclanespectra} that the mod $p$ symmetric powers of the equivariant sphere spectrum are a filtration for the equivariant Eilenberg-Maclane spectrum of $\underline{\F}_p$. Written below are only the stages at powers of $p$,
$$\Sigma^{\infty G}S^0=\Sp^1(\Sigma^{\infty G}S^0) \subset \Sp^p(\Sigma^{\infty G}S^0) \subset \Sp^{p^2}(\Sigma^{\infty G}S^0)\subseteq \cdots\subseteq \Sp^{\infty}(\Sigma^{\infty G}S^0)\simeq H\underline{\F}_p.$$
In this section, we give a short proof of Theorem \ref{thm:eqsplitting}
\setcounter{theorem}{1}
\begin{theorem}
Let $G$ be any finite $p$ group. The filtration $\{\Sp_{\Z/p}^{p^n}(\Sigma^{\infty G}S^0)\}_{n\ge 0}$ splits into its layers after smashing with $H\underline{\F}_p$. That is, there is an equivalence of $H\underline{\F}_p$-modules
$$H\underline{\F}_p\sm H\underline{\F}_p\simeq \bigvee\limits_{n\ge 0}H\underline{\F}_p\sm \Sigma^n \mathbf{e}_nB_G(\Z/p)^n_+.$$
\end{theorem}
\setcounter{theorem}{71}
First, we need a lemma.

\begin{lemma}\label{lemma:splitfirststage}
The inclusion $\Sigma^{\infty G}S^0 \rightarrow \Sp_{\Z/p}^p(\Sigma^{\infty G}S^0)$ has a retraction after smashing with $H\underline{\F}_p$, namely
$$\xymatrix{\Sigma^{\infty G}S^0\sm H\underline{\F}_p \ar[r]& \Sp_{\Z/p}^p(\Sigma^{\infty G}S^0)\sm H\underline{\F}_p\ar@{-->}@/_1pc/[l]}.$$
\end{lemma}
\begin{proof}
The composition
$$\xymatrix{\Sigma^{\infty G}S^0 \ar[r] & \Sp_{\Z/p}^p(\Sigma^{\infty G}S^0)\ar[r] & \Sp_{\Z/p}^{\infty}(\Sigma^{\infty G}S^0)\simeq H\underline{\F}_p}$$
is the unit map. Therefore, after smashing with $H\underline{\F}_p$, and composing with the product map $\mu:H\underline{\F}_p \sm H\underline{\F}_p \rightarrow H\underline{\F}_p$, we have the commutative diagram
$$\xymatrix{\Sigma^{\infty G}S^0\sm H\underline{\F}_p\ar[r]\ar[rrd]_{\text{Id}} & (\Sp_{\Z/p}^p(\Sigma^{\infty G}S^0)\sm H\underline{\F}_p \ar[r]\ar@{-->}[rd] & H\underline{\F}_p\sm H\underline{\F}_p\ar[d]^{\mu}\\
&&H\underline{\F}_p}$$
The dotted map provides the splitting.
\end{proof}
Now Theorem \ref{thm:eqsplitting} is an immediate corollary of the following proposition.
\begin{exe}
There is a monomorphism (dotted) which splits the cofiber sequence shown.

$$\xymatrix{H\underline{\F}_p\sm \Sp_{\Z/p}^{p^{n-1}}(\Sigma^{\infty G}S^0)\ar[r] & H\underline{\F}_p\sm \Sp_{\Z/p}^{p^n}(\Sigma^{\infty G}S^0)\ar[r] & H\underline{\F}_p\sm \Sp_{\Z/p}^{p^n}(\Sigma^{\infty G}S^0)/\Sp_{\Z/p}^{p^{n-1}}(\Sigma^{\infty G}S^0)\ar@{-->}@/_1pc/[l]}$$
\end{exe}
\begin{proof}
Let $t_1$ denote the retraction map of Lemma \ref{lemma:splitfirststage}, i.e.
$$\xymatrix{\Sp_{\Z/p}^p(\Sigma^{\infty G}S^0)\sm H\underline{\F}_p \ar[r]&\Sp_{\Z/p}^p(\Sigma^{\infty G}S^0)/\Sp_{\Z/p}^1(\Sigma^{\infty G}S^0)\sm H\underline{\F}_p\ar@{-->}@/_1pc/[l]_{t_1}}.$$
Define the map $t_n$ by the composite in the top half of the diagram of $H\underline{\F}_p$-module spectra below,
$$\xymatrix{\Sp_{\Z/p}^{p^n}(\Sigma^{\infty G}S^0)/\Sp_{\Z/p}^{p^{n-1}}(\Sigma^{\infty G}S^0)\sm H\underline{\F}_p \ar@{=}[r]\ar@{-->}[dd]_{t_n} & \Sigma^n \mathbf{e}_nB_G(\Z/p)_+^n \sm H\underline{\F}_p\ar[d]^{\subset}\\
& (\Sigma \mathbf{e}_1B_G\Z/p_+)^{\sm n}\sm H\underline{\F}_p\ar[d]^{t_1^{\sm n}}\\
\Sp_{\Z/p}^{p^n}(\Sigma^{\infty G}S^0)\sm H\underline{\F}_p\ar[d] & (\Sp_{\Z/p}^p(\Sigma^{\infty G}S^0))^{\sm n}\sm H\underline{\F}_p\ar[l]\ar[d]\\
\Sp_{\Z/p}^{p^n}(\Sigma^{\infty G}S^0)/\Sp_{\Z/p}^{p^{n-1}}(\Sigma^{\infty G}S^0)\sm H\underline{\F}_p & \Sp_{\Z/p}^{p^n}(\Sigma^{\infty G}S^0)\sm H\underline{\F}_p\ar[l]}.$$
Since $t_1$ is a retraction, it follows that the composite map
$$\Sigma^n\mathbf{e}_nB_G(\Z/p)_+^n\sm H\underline{\F}_p \rightarrow \Sp_{\Z/p}^{p^n}(\Sigma^{\infty G}S^0)/\Sp_{\Z/p}^{p^{n-1}}(\Sigma^{\infty G}S^0)\sm H\underline{\F}_p$$
along the right and bottom of the commutative diagram, is an equivalence. Therefore, by commutativity, the composite along the left side of the diagram is an equivalence, and so $t_n$ is a retraction.
\end{proof}


\begin{thebibliography}{1}
\bibitem{AM} A. Adem, J. Milgram, \emph{Cohomology of Finite Groups}, Springer-Verlag Grundlehren \textbf{309} (2004)

\bibitem{AD} G. Arone, W. Dwyer, \emph{Partition Complexes, Tits Buildings and Symmetric Products}, Proceedings of the London Mathematical Society (3) \textbf{82} (2001), no. 1, pp. 229-256

\bibitem{ADL} G. Arone, W. Dwyer, K. Lesh, \emph{Bredon Homology of Partition Complexes}, Documenta Mathematica \textbf{21} (2016), pp. 1227-1268







\bibitem{DS} P. Dos Santos, \emph{A Note on the Equivariant Dold-Thom Theorem}, Journal of Pure and Applied Algebra \textbf{183} (2003), No. 1-3, pp. 299-312

\bibitem{MaG} J. P. C. Greenlees, J. P, May, \emph{Equivariant Stable Homotopy theory}, in Ioan James (ed.), \emph{Handbook of Algebraic Topology}, pp. 279-325 (1995)

\bibitem{GMM} B. Guillou, J. P. May, M. Merling, \emph{Categorical Models for Equivariant Classifying Spaces}, Algebraic and Geometric Topology \textbf{17} (2017), No. 5, pp. 2565-2602



\bibitem{HHR} M. Hill, M. Hopkins, D. Ravenel, \emph{On the Non-existence of Elements of Kervaire Invariant One}, Annals of Mathematics \textbf{184} (2016), No. 1, pp. 1-262

\bibitem{HK} P. Hu, I. Kriz, \emph{Real-oriented homotopy theory and an analogue of the Adams-Novikov spectral sequence}, Topology \textbf{40} (2001), 317-399


\bibitem{Ku} N. J. Kuhn, \emph{A Kahn-Priddy sequence and a conjecture of G. W. Whitehead}, Mathematical Proceedings of the Cambridge Philosophical Society \textbf{92}, 467-483 (1982)

\bibitem{Ku2} N. J. Kuhn, \emph{The Whitehead conjecture, the Tower of $S^1$ conjecture, and Hecke Algebras of type A}, Journal of Topology \textbf{8} (2015), Issue 1, pp. 118-146

\bibitem{Ku3} N. J. Kuhn, \emph{Spacelike resolutions of Spectra}, Proceedings of the Northwestern Homotopy Theory Conference (1982), Evanston, IL


\bibitem{KMP} N. J. Kuhn, S. Mitchell, S. Priddy, \emph{The Whitehead Conjecture and Splitting $B(\Z/2)^k$}, Bulletin of the American Mathematical Society \textbf{7} 255-258 (1982)

\bibitem{LF} P.C. Lima-Filho, \emph{On the equivariant homotopy of free abelian groups on $G$-spaces and $G$-spectra}, Mathematische Zeitschrift \textbf{224} (1997), 567-601

\bibitem{LSWX} G. Li, X. D. Shi, G. Wang, Z. Xu, \emph{Hurewicz Images of Real Bordism Theory and Real Johnson-Wilson Theories}, Advances in Mathematics \textbf{342} (2019), pp. 67-115

\bibitem{MC} M. C. McCord, \emph{Classifying Spaces and Infinite Symmetric Products}, Transactions of the American Mathematical Society, \textbf{146} (1969), 273-298


\bibitem{Mit} S. Mitchell, \emph{Finite complexes with $A(n)$-free cohomology}, Topology, \textbf{24} (1985), No. 2, pp. 227-248

\bibitem{MitCF} S. Mitchell, \emph{A Proof of the Conner-Floyd Conjecture}, American Journal of Mathematics \textbf{106} (1984), No. 4, pp. 889-891

\bibitem{MP} S. Mitchell, S. Priddy, \emph{Stable Splittings Derived from the Steinberg Module}, Topology \textbf{22} (1983), Issue 3, 285-298

\bibitem{Na} M. Nakaoka, \emph{Cohomology mod $p$ of symmetric products of spheres}, Journal of the Institute of Polytechnics Osaka City University Series A \textbf{9} (1958), No. 1, 1-18

\bibitem{Sankar} K. Sankar, \emph{Equivariant Steinberg Summands}, arXiv:1903.08246


\bibitem{Wel} P. Welcher, \emph{Symmetric fibre spectra and $K(n)$-homology acyclicity}, Indiana University Mathematics Journal \textbf{30} (1981), no. 6, 801-812


\end{thebibliography}
\end{document}